\documentclass[11pt,
]
{amsart}%
\usepackage[hypertex]{hyperref}
\usepackage{cite}
\usepackage{amssymb}
\usepackage{amsmath}
\usepackage{amsthm}
\usepackage{amsfonts}
\usepackage{amscd}
\usepackage[toc,page]{appendix}

\newtheorem{thm}{Theorem}[section]
\newtheorem{cor}[thm]{Corollary}
\newtheorem{ex}[thm]{Example}

\newtheorem{lem}[thm]{Lemma}

\newtheorem{prop}[thm]{Proposition}

\newtheorem{as}[thm]{Assumption}

\theoremstyle{definition}
\newtheorem*{acknowledgement}{Acknowledgement}
\newtheorem{df}[thm]{Definition}

\newtheorem{rem}[thm]{Remark}

\numberwithin{equation}{section}

\begin{document}

\title[Harnack inequalities in infinite dimensions]{Harnack inequalities in infinite dimensions}
\author[Bass]{Richard F. Bass{$^{\dagger }$}}
\thanks{\footnotemark {$^\dagger$}This research was supported in part by NSF Grant DMS-0901505.}
\address{Department of Mathematics\\
University of Connecticut\\
Storrs, CT 06269, U.S.A. } \email{r.bass@uconn.edu}
\author[Gordina]{Maria Gordina{$^{*}$}}
\thanks{\footnotemark {$*$}This research was supported in part by NSF Grant DMS-1007496.}
\address{Department of Mathematics\\
University of Connecticut\\
Storrs, CT 06269, U.S.A. } \email{maria.gordina@uconn.edu}


\date{\today }

\begin{abstract}
We consider the Harnack inequality for harmonic functions with
respect to three types of infinite-dimensional operators.
For the infinite dimensional Laplacian, we show no Harnack inequality is possible.
We also show that the Harnack inequality fails for a large class of
Ornstein-Uhlenbeck processes, although functions that are harmonic with
respect to these processes do satisfy an \emph{a priori} modulus of continuity.
Many of these processes also have a coupling property. The third type
of operator considered is the infinite dimensional analog of operators
in H\"ormander's form. In this case a Harnack inequality does hold.
\end{abstract}

\maketitle

\section{Introduction\label{s.1}}

The Harnack inequality is an important tool in analysis, partial
differential equations, and probability theory. For over half
a century there has been intense interest in extending the Harnack
inequality to more general operators than the Laplacian, with
seminal papers by Moser \cite{Moser1961a} and Krylov-Safonov
\cite{KrylovSafonov1980}. See \cite{Kassmann2007a} for a survey of
some recent work.

It is a natural question to ask whether the Harnack inequality holds
for infinite-dimensional operators. If ${L}$ is an infinite
dimensional operator and $h$ is a function that is non-negative and
harmonic in a ball  with respect to the operator ${L}$
and $B_2$ is a ball with the same center as $B_1$ but of smaller
radius, does there exist a constant $c$ depending on $B_1$ and $B_2$ but
not on $h$ such
that
$$h(x)\leqslant ch(y)$$
for all $x,y\in B_2$?

When one considers the infinite-dimensional Laplacian, or alternatively
the infinitesimal generator of infinite-dimensional Brownian motion,
there is first the question of what one means by a ball. In this case
there are two different norms present, one for a Banach space and one
for a Hilbert space. We show that no matter what combination of definitions
for $B_1$ and $B_2$ that are used, no  Harnack inequality is possible.
Our technique is to use estimates for Green functions for finite
dimensional Brownian motions and then to go from there to the infinite
dimensional Brownian motion.

For more on the potential theory of infinite-dimensional Brownian motion
we refer to the classic work of L. Gross \cite{Gross1967a}, as well as to \cite{Elson1974, Piech1972a, Piech1972b, KuoBook1975, Carmona1980, BogachevBook}. V. Goodman \cite{Goodman1972a, Goodman1973} has several interesting papers on harmonic functions for the infinite-dimensional Laplacian.

We next turn to the infinite-dimensional Ornstein-Uhlenbeck process and its infinitesimal generator. See \cite{KuoBook1975, ShigekawaBook, DaPratoZabczykBook3} for the construction and properties of these processes.
In this case, the question of the definitions of $B_1$ and $B_2$ is not an issue.

We show that again, no Harnack inequality is possible.  We again use
estimates for the Green functions of finite dimensional approximations,
but unlike in the Brownian motion case, here the estimates are quite delicate.

We also establish two positive results for a large class of infinite
dimensional Ornstein-Uhlenbeck processes. First we show that functions
that are harmonic in a ball are continuous and satisfy an \emph{a priori}
modulus of continuity.

Secondly, it is commonly thought that there is a close connection
between coupling and the Harnack inequality. See \cite{BarlowBass1999a}
for an example where this connection is explicit. By coupling,
we mean that given $B_2\subset B_1$ with the same center but different
radii and $x,y\in B_2$, it is possible to construct two Ornstein-Uhlenbeck
processes  $X$ and $Y$ started at $x,y$, resp.\
(by no means independent),
such that the two processes meet (or couple) before either process exits
$B_1$. Even though the Harnack inequality does not hold, we show that for
a large class of Ornstein-Uhlenbeck processes it is possible to
establish a coupling result.

Finally we turn to the infinite-dimensional analog of operators
in H\"or\-man\-der's form. These are operators of the form
$${L} f(x)=\sum_{j=1}^n \nabla_{A_j}^2 f(x),$$
where $\nabla_{A_j}$ is a smooth vector field. For these
operators we are able to establish a Harnack inequality.
To define a ball in this context we use a distance intimately tied
to the vector fields $A_1, \ldots, A_n$. In addition, we connect this distance to another distance introduced in \cite{BiroliMosco1991} for Dirichlet forms, and later used in connection with parabolic Harnack inequalities in different settings in \cite{SaloffCoste1995a}.

Our technique to prove the Harnack inequality for these operators in H\"{o}rmander's form is to employ methods developed by Bakry, {\'E}mery, and Ledoux. For general reviews on their approach with applications to functional inequalities see \cite{Bakry2006Tata, Ledoux2000a}. We prove a curvature--dimension inequality,
derive a Li-Yau estimate from that, and then prove a parabolic
Harnack inequality, from which the usual Harnack inequality  follows. For this approach on Riemannian manifolds with Ricci curvature bounded below we refer to \cite{BakryQian1999}.

We are not the first to investigate Harnack inequalities for infinite
dimensional operators. In addition to the papers
\cite{BiroliMosco1991} and  \cite{Berg1976} mentioned above, they have been investigated by Bendikov and
Saloff-Coste \cite{BendikovSaloffCoste2000}, who studied the
related potential theory as well. Their
context is quite different from ours, however, as they consider infinite-dimensional spaces which are close to finite-dimensional spaces, such as infinite products of tori. This allows them to modify some of the techniques used for finite
dimensional spaces.

We mention three open problems that we think are of interest:

1. Our positive result is for operators that are the infinite-dimensional
analog of H\"{o}rmander's form, but we only have a finite number of vector
fields. The corresponding processes need not live in any finite dimensional
Euclidean space, but one would still like to allow the
possibility of there being infinitely many
vector fields.

2. Are there any infinite-dimensional processes of the form Laplacian plus drift for which
a Harnack inequality holds?

3. Restricting attention to the infinite-dimensional Ornstein-Uhlenbeck process,
can one define $B_1$ and $B_2$ in terms of some alternate definition
of distance such that the Harnack inequality holds?

The outline of our paper is straightforward. Section \ref{s.3}
considers infinite-dimensional Brownian motion, Section \ref{S:OU}
contains our results on infinite Ornstein-Uhlenbeck processes, while
our Harnack inequality for operators of H\"{o}rmander form appears
in Section \ref{S:hor}.

We use the letter $c$ with or without subscripts for finite
positive constants whose exact value is unimportant and which
may change from place to place.

\begin{acknowledgement}
We are grateful to Leonard Gross and Laurent Saloff-Coste for providing us with necessary background on the subject. Our thanks also go to Bruce Driver, Tai Melcher and Sasha Teplyaev for stimulating discussions.
\end{acknowledgement}

\section{Brownian Motion}\label{s.3}

We first prove a proposition that contains the key idea.
Let $B^{(n)}(x,r)=\{y\in \mathbb{R}^n: |x-y|<r\}$, where $|x-y|=
\Big(\sum_{i-1}^n |x_i-y_i|^2\Big)^{1/2}$.

\begin{prop}\label{rb-s2-P1}
Let $K>0$. For all $n$ sufficiently large, there exists a
function $h_n$ which is non-negative and harmonic on its  domain $B^{(n)}(0,1)$
and points $x_n,z_n\in B^{(n)}(0,1/2)$ such that

\[
\frac{h_n(z_n)}{h_n(x_n)}\geqslant K.
\]
\end{prop}

\begin{proof} Let $G_n(x,y)=|x-y|^{2-n}$, a constant multiple of the
Newtonian potential density on $\mathbb{R}^n$. Let $e_1=(1,0, \ldots, 0)$.
If we set $h_n(x)=G_n(x,e_1)$, then it is well-known that $h_n$ is harmonic
in $\mathbb{R}^n\setminus\{0\}$.

Let $x_n=0$ and $z_n=\frac14 e_1$. Both are in $B^{(n)}(0,1/2)$
and
\[
\frac{h_n(z_n)}{h_n(x_n)}=\frac{(3/4)^{2-n}}{1^{2-n}}\geqslant K
\]
if $n$ is sufficiently large.
\end{proof}

Next we embed the above finite-dimensional example into the framework of infinite-dimensional
Brownian motion.

Let $(W,H,\mu)$ be an abstract Wiener space, where $W$ is a separable
Banach space, $H$ is a Hilbert space, and $\mu$ is a Gaussian measure.
For background about abstract Wiener spaces, see \cite{BogachevBook}
or \cite{KuoBook1975}.
We use $\|\cdot\|_H$ and $\|\cdot\|_W$ for the norms on $H$ and $W$,
respectively.
We denote the inner product on $H$ by $\langle \cdot, \cdot \rangle_H$.

The classical example of an abstract Wiener space has $W$ equal to the
continuous functions on $[0,1]$ that are 0 at 0 and has $H$ equal to the
functions in $W$ that are absolutely continuous and whose derivatives
are square integrable. Another example that perhaps better illustrates what follows is
to let $H$ be the set of sequences $(x_1, x_2, \ldots)$ such that
$\sum_i x_i^2<\infty$ and let $W$ be the set of sequences such that
$\sum_i \lambda_i^2 x_i^2<\infty$, where $\{\lambda_i\}$ is a fixed
sequence with $\sum_i \lambda_i^2<\infty$.

Let $ H_\ast$ be the set of $h\in H$ such that
 $\left\langle \cdot
,h\right\rangle _{H}\in H^{\ast}$ extends to a continuous linear functional on
$W.$  Here $H^*$ is the dual space of $H$, and is, of course, isomorphic to $H$.
(We will continue to denote the continuous extension of $\left\langle
\cdot,h\right\rangle _{H}$ to $W$ by $\left\langle \cdot,h\right\rangle
_{H}.)$

Next suppose that $P:H\rightarrow H$ is a finite rank orthogonal projection
such that $PH\subset H_{\ast}$. Let $\left\{  e_{j}\right\}_{j=1}^{n}$ be an
orthonormal basis for $PH$ and $\ell_{j}=\left\langle \cdot,e_{j}\right\rangle
_{H}\in W^{\ast}.$ Then we may extend $P$ to a unique continuous operator
from $W$ $\rightarrow H$ (still denoted by $P)$ by letting
\begin{equation}
Pw:=\sum_{j=1}^{n}\left\langle w, e_{j}\right\rangle _{H}e_{j}=\sum_{j=1}%
^{n}\ell_{j}\left(  w\right)  e_{j}\text{ for all }w\in W. \label{e.2.17}
\end{equation}
For more details on these projections see \cite{DriverGordina2008}.

Let $\operatorname*{Proj}\left(  W\right)  $ denote the
collection of finite rank projections on $W$ such that $PW\subset H_{\ast}$
and $P|_{H}:H\rightarrow H$ is an orthogonal projection, i.e. $P$ has the form
given in \eqref{e.2.17}. As usual a function $f:W \rightarrow\mathbb{R}$ is  a
(smooth) cylinder function if it may be written as
$f=F\circ P$ for some $P\in\operatorname*{Proj}\left(  W\right)$ and
some (smooth) function $F:\mathbb{R}^{n} \rightarrow \mathbb{R}$, where $n$ is the rank of $P$. For example, let $\{ e_{n} \}_{n=1}^{\infty}$ be an orthonormal basis of $H$ such that $e_{n} \in H_{\ast}$, and $H_n$ be the span of $\{e_1, \ldots, e_n\}$ identified with $\mathbb{R}^{n}$. For each $n$, define $P_{n}\in \operatorname*{Proj}\left(  W\right)$
by
\[
P_{n} : W \rightarrow H_{n}  \subset H_{\ast} \subset H
\]
as in \eqref{e.2.17}.

For $t \geqslant 0$ let $\mu_{t}$ be the rescaled measure $\mu_{t}\left( A \right): = \mu_{t}\left( A/\sqrt{t} \right)$ with $\mu_{0}=\delta_{0}$. Then as was first noted by  Gross in \cite[p. 135]{Gross1967a} there  exists a stochastic process $B_{t}, t \geqslant 0$, with values in $W$ which is  continuous
a.s.\ in $t$ with respect to the norm topology on $W$, has independent increments, and for $s<t$
has $B_{t}-B_{s}$ distributed as $\mu_{t-s}$, with $B_{0}=0$ a.s. $B_{t}$ is called standard Brownian motion on $\left( W, \mu \right)$.

Let $\mathcal{B}(W)$ be the Borel $\sigma$-algebra on $W$. If we set
$ \mu_{t}\left( x, A \right):=\mu_{t}\left( x- A \right)$, for $A \in \mathcal{B}\left( W \right),$
then it is well known that $\{\mu_t\}$ forms a family of Markov transition
kernels, and we may thus view $(B_t, \mathbb{P}^x)$ as a strong Markov
process with state space $W$, where $\mathbb{P}^x$ is the law of
$x+B$. We do not need this fact in what follows, but want to point out that $B_{n}\left(  t\right):=P_{n}B\left(  t\right)  \in P_{n}H\subset H\subset W$ give a natural approximation to $B\left( t \right)$ as is pointed out in  \cite[Proposition 4.6]{DriverGordina2008}.

We denote the open ball in $W$ of radius $r$ centered at $x \in W$ by $B(x,r)$ and its boundary by $S_{r}\left( x \right)$. The first exit time of $B_t$ from $B(0,r)$ will be denoted by $\tau_r$. By \cite[Remark 3.3]{Gross1967a} the exit time $\tau_ r$ is finite a.s.

A set $E$ is open in the fine topology if for each $x\in E$ there exists
a Borel set $E_x\subset E$ such that $\mathbb{P}(\sigma_{E_x}>0)=1$,
where $\sigma_{E_x}$ is the first exit from $E_x$.

Let $f$ be  a locally bounded,  Borel measurable, finely continuous, real-valued function $f$ whose domain is an open set in $W$. Then $f$ is harmonic if
\begin{equation}\label{e.4.1}
    f\left( x \right)=
\int_{S_{r}\left( 0 \right)} f\left( x+y \right) \pi_{r}\left(  dy\right)
\end{equation}
    for any $r$ such that  the closure of $B(x,r)$ is contained in the domain of $f$,
where
\[
\pi_r(dy)=\mathbb{P}^0(B_{\tau_{r}}\in \, dy).
\]

 Let $f$ be a real-valued function on $W$. We can consider $F\left( h \right)=f\left( x+h\right)$ as a function on $H$. If $F$ has the Fr\'{e}chet derivative at $0$, we say that $f$ is {$H$-differentiable}. Similarly we can define the second $H$-derivative $D^{2}$, and finally
\[
\Delta f\left( x \right):=\operatorname{tr} D^{2} f \left( x \right)
\]
whenever $D^{2} f \left( x \right)$ exists and of trace class.

The following properties can be found in \cite[Theorems 1, 2, 3]{Goodman1972a}
\begin{thm}\label{t.4.9} Let $\left( W, H, \mu \right)$ be an abstract Wiener space.\\
(1) A harmonic function on $W$ is infinitely $H$-differentiable. The second derivative of a harmonic function at each point of its domain is a Hilbert-Schmidt operator.\\
(2) If a harmonic function on $W$ satisfies a uniform Lipschitz condition in a neighborhood of a point $x$, then the Laplacian of $u$ exists at $x$ and $\left( \Delta u \right)\left( x \right)=0$.
\end{thm}

\begin{rem}\label{rb-s2-R1}{\rm So far the theory of harmonic functions
in infinite dimensions may not seem that different from the
finite dimensional case. There are, however, striking differences.
For example, Goodman \cite[Proposition 4]{Goodman1972a} shows there exists
 a harmonic function
that is not continuous with respect to the topology of $W$. In view of the
previous theorem, however, it is smooth with respect to the topology of
$H$.
}
\end{rem}

Let $\left( W, H, \mu\right)$ be an abstract Wiener space. Denote by $G_{n}\left( x, z \right)$ the function on $\mathbb{R}^n\times \mathbb{R}^n$ defined
by $G_n(x,z)=\vert x-z \vert^{2-n}$. Consider $P_{n} \in \operatorname{Proj}\left( W \right)$ as defined by \eqref{e.2.17}, and define the  cylinder function $g_{n}\left( w \right):=G_{n}\left( P_{n}w, P_{n} z \right)$ for any $w \in W$ and $z=e_{1}$.

\begin{prop}
\label{mg-s2-P2}
The function $g_{n}$ is harmonic on $W$ away from the set $\{ w\in W: P_{n}w=e_{1}\}=\{ w\in W: e_{1}\left( w \right)=1\}$.
\end{prop}

\begin{proof}
We need to check that $g_{n}$ is locally bounded, Borel measurable, finely
continuous,  and \eqref{e.4.1} holds with $f$ replaced by $g_n$  for all $r>0$
whenever the closure of $B_{r}\left( x \right)$ is contained in the domain of $g_{n}$. One can show that $g_{n}$ is locally bounded, Borel measurable, and finely continuous similarly to \cite[p.~455]{Goodman1972a}.

Now we check the last part. Suppose $x \notin \{ w\in W: P_{n}w=e_{1}\}$.

\begin{align*}
 \int_{S_{r}\left( 0 \right)} g_{n}\left( x+y \right)\pi_{r}\left( dy\right)&=\int_{S_{r}\left( 0 \right)} G_{n}\circ P_{n}\left( x+y \right)\pi_{r}\left( dy\right)
\\
&= \mathbb{E}^x\left( G_{n}\circ P_{n}\left( B_{\tau_{r}} \right)\right)\\
&=\mathbb{E}^x\left( G_{n}\circ P_{n}\left( P_{n} B_{\tau_{r}} \right)\right).
\end{align*}
Note that $P_{n}B_{t}$ is a martingale, and $\tau_{r}$ is a stopping time, and we would like to use the optional stopping time theorem. We need to point out here that $e_{1} \in H_{\ast}\subset H$ and therefore $P_{n}e_{1}=e_{1}$. So if we choose $r<{1}/2\Vert e_{1} \Vert_{W^{\ast}}$,  then $e_{1}\notin P_{n} B(0,r)$. Indeed, if there is a $w \in B(0,r)$ such that $P_{n}w=e_{1}$, then $e_{1}\left( w \right)=\langle w, e_{1} \rangle=1$. But
\[
\vert e_{1}\left( w \right)\vert \leqslant \Vert e_{1} \Vert_{W^{\ast}} \Vert w \Vert_{W}<r\Vert e_{1} \Vert_{W^{\ast}}<\frac{1}{2}
\]
which is a contradiction. Thus $G_{n}$ is harmonic in $P_{n} B(0,r)\subseteq P_{n}H \cong \mathbb{R}^{n}$ and therefore
\[
\int_{S_{r}\left( 0 \right)} g_{n}\left( x+y \right)\pi_{r}\left( dy\right)=G_{n}\left( P_{n}x \right)=g_{n}\left( x\right).
\]

\end{proof}
Our main theorem of this section is now simple.

\begin{thm}\label{rb-s2-T1} For each $n$ there exist functions  $g_n$
that are non-negative and harmonic in the ball of radius 1 about 0 with
respect to the norm of $W$ and points $x, z$ in the ball of radius
$1/2$ about 0 with respect to the norm of $H$ such that

\[
\frac{g_n(z)}{g_n(x)}\to \infty
\]
as $n\to \infty$. In particular, the Harnack inequality fails.
\end{thm}

\begin{proof}
We let $g_n$ be as above and $x=0$ and $z=\frac14 e_1$ for all
$n$. Our result follows by combining Propositions \ref{rb-s2-P1} and
\ref{mg-s2-P2}.
\end{proof}

\section{Ornstein-Uhlenbeck process}\label{S:OU}

Let $H$ be a separable Hilbert space with inner product $\langle \cdot, \cdot \rangle$ and corresponding norm $\vert \cdot \vert$.
Define
\[ \Vert f \Vert_{0}:=\sup_{x \in H} \vert f\left( x \right)\vert.  \]

Recall (see \cite{DaPratoZabczykBook3})
that
for an arbitrary positive trace class operator $Q$ on $H$ and  $a \in H$ there exists a unique measure $N_{a, Q}$ on $\mathcal{B}\left( H \right)$ such that
\[
\int_{H}e^{i\langle h, x \rangle}N_{a, Q}\left( dx \right)=e^{i\langle a, h\rangle-\frac{1}{2}\langle Qh, h \rangle},
\qquad h\in H.
\]
We call such $N_{a, Q}\left( dx \right)$ a Gaussian measure with mean $a$ and covariance $Q$. It is easy to check that
\[
\int_{H}  x  N_{a, Q}\left( dx \right)=a,
\]
\[\int_{H} \vert x-a \vert^{2} N_{a, Q}\left( dx \right)=\operatorname{Tr} Q,
\]
\[\int_{H} \langle x-a, y \rangle \langle x-a, z \rangle N_{a, Q}\left( dx \right)=\langle Qy, z \rangle,\qquad \mbox{\rm and}
\]
\[\frac{dN_{b, Q}}{dN_{a, Q}}\left( dy \right)=e^{-\frac{1}{2}\vert Q^{-1/2}\left( a-b \right)\vert^{2}+\langle Q^{-1/2}\left( y-a\right), Q^{-1/2}\left( b-a\right)\rangle }.
\]

We consider the Ornstein-Uhlenbeck process in a separable Hilbert space $H$. The process in question is a solution to the stochastic differential equation
\begin{equation}\label{e.8.1}
dZ_{t} = -AZ_{t} dt + Q^{1/2}\,dW_{t}, \qquad Z_{0}=x,
\end{equation}
where $A$ is the generator of a strongly continuous semigroup $e^{-At}$ on $H$, $W$ is a cylindrical Wiener process on $H$, and $Q: H \to H$ is a positive bounded operator. The solution to \eqref{e.8.1} is given by
\[
Z_{t}^{x} = e^{-At}x + \int_{0}^{t} e^{-A\left( t -s \right)}Q^{1/2}\, dW_{s}.
\]
The corresponding transition probability is defined as usual by
$\left(P_{t}f\right)\left( x \right)=\mathbb{E}f\left( Z_{t}^{x}\right)$, $f \in \mathcal{B}_{b}\left( H \right),$
where $\mathcal{B}_b(H)$ are the bounded Borel measurable functions on $H$.
It is known that  the law of $Z_{t}$ is  a Gaussian measure centered at $e^{-At}x$ with covariance
\[
Q_{t} =
\int_{0}^{t} e^{-A\left( t -s \right)}Q e^{-A^{\ast}\left( t -s \right)}\,ds,
\]
which we called $N_{e^{-tA}x, Q_{t}}\left( dy \right)$. Note that for the corresponding parabolic equation in $H$ to be well-posed we need a basic assumption on $Q_{t}$ to be non-negative and trace-class for all $t>0$ \cite[p. 99]{DaPratoZabczykBook3}.

We assume the controllability condition
\begin{equation}\label{e.8.2}
e^{-At}\left( H \right) \subset Q_{t}^{1/2}\left( H \right) \text{ for all } t>0
\end{equation}
holds.
As is described in \cite[p. 104]{DaPratoZabczykBook3}, under the condition \eqref{e.8.2} the stochastic differential equation in question has a classical solution.
We   define
\begin{equation}\label{e.5.3}
\Lambda_{t}: = Q_{t}^{-1/2}e^{-tA}, \qquad t>0,
\end{equation}
where $Q_{t}^{-1/2}$ is the pseudo-inverse of $Q_{t}^{1/2}$. By the closed
graph theorem we see that $\Lambda_{t}$ is a bounded operator in $H$ for all $t > 0$.

Suppose $Q=I$, the identity operator, and $A$ is a self-adjoint invertible operator on $H$, then
\[
Q_t =\int_0^t e^{-sA}e^{-sA^*} x ds=\int_0^t e^{-2sA} ds=\frac{1}{2} A^{-1}(I-e^{-2tA}),  t\geqslant 0.
\]
If in addition we assume that $A^{-1}$ is trace-class, then there is an orthonormal basis $\{e_n\}_{n=1}^{\infty}$ of $H$ and the corresponding eigenvalues $a_n$ such that
\[
Ae_n=a_n e_n, a_n>0, a_n\uparrow \infty, \sum_{n=1}^{\infty} a_n^{-1}<\infty.
\]
Then $Q_t$  is diagonal in the orthonormal basis $\{e_n\}_{n=1}^{\infty}$:
\[
Q_te_n=\frac{t\left( e^{2ta_n}-1 \right)}{2ta_ne^{2ta_n}}e_n.
\]
Then  $Q_t$ is trace class with
\[
\operatorname{Tr}Q_{t}=\sum_{n=1}^{\infty} \frac{t\left( e^{2ta_n}-1 \right)}{2ta_ne^{2ta_n}} \leqslant
\sum_{n=1}^{\infty} \frac{1}{2a_n}=\frac{\operatorname{Tr} A^{-1}}{2}<\infty.
\]
Now we see that
\[
\Lambda_{t}e_{n}=\frac{\sqrt{2ta_n}}{t^{1/2}\sqrt{e^{2ta_n}-1}}e_n,
\]
and so
$\vert \Lambda_{t}x \vert\leqslant \vert x \vert/\sqrt{t}$. This proves the following proposition.

\begin{prop}\label{p.3.1} Assume that $Q=I$ and $A^{-1}$ is trace-class.
Then the operator $Q_t$ is a trace-class operator on $H$ and
$\Vert \Lambda_{t} \Vert \leqslant 1/\sqrt{t}$.
\end{prop}

Using the properties of Gaussian measures,  we see that the the Ornstein-Uhlenbeck semigroup can be described by the following Mehler formula
\begin{equation}\label{rb-sec3.2-E21}
\left(P_{t}f\right)\left( x \right)=\int_{H}f\left( z+e^{-tA}x\right) N_{0, Q_{t}}\left( dz\right).
\end{equation}

\subsection{Modulus of continuity for harmonic functions}

We establish an \emph{a priori}
 modulus of continuity for harmonic functions.

\begin{lem}\label{rb-s3-T51} Suppose \eqref{e.8.2} is satisfied. If $f$
is a bounded Borel measurable function on $H$ and $t>0$, there exists
a constant $c(t)$ not depending on $f$ such that
\begin{equation}\label{rb-s3-E1}
|P_tf(x)-P_tf(y)|\leqslant c\|f\|_0|x-y|, \qquad x,y\in H.
\end{equation}
Moreover, for any $u \in H$

\[
 D_{u}P_{t}f\left( x\right) \leqslant \left( P_{t}f^{2}\left( x \right)\right)^{1/2}\Vert \Lambda_{t}u\Vert^{2}.
\]

\end{lem}

\begin{proof}
Consider $N_{0, Q_{t}}\left( dz \right)$, a centered Gaussian measure with covariance $Q_{t}$. By the Cameron-Martin theorem the
transition probability
$$P_{t}^{x}\left( dz \right)=N_{e^{-tA}x, Q_{t}}\left( dz \right)$$
 has a density with respect to $N_{0, Q_{t}}\left( dz \right)$  given by
\begin{equation}\label{e.4.4}
J_{t}\left( x, z\right):=\frac{N_{e^{-tA}x, Q_{t}}\left( dz \right)}{N_{0, Q_{t}}\left( dz \right)}=\exp\left( \langle \Lambda_{t}x, Q_{t}^{-1/2} z\rangle-\frac{1}{2}\vert \Lambda_{t}x\vert ^{2}\right).
\end{equation}
Thus
\begin{equation}\label{e.5.2}
\left( P_{t}f\right)\left( x\right)=\int_{H} J_{t}\left( x, z\right) f\left( z \right) N_{0, Q_{t}}\left( dz \right).
\end{equation}
Now we can use \eqref{e.5.2} to estimate the derivative $D_{u}P_{t}f$ for any $u \in H$
by
\begin{align*}
 D_{u}P_{t}f\left( x\right)&=\int_{H}\langle \Lambda_{t}u, Q_{t}^{-1/2} \left( z-e^{-At}x\right)\rangle f\left( z \right)J_{t}\left( x, z \right) N_{0, Q_{t}}\left( dz \right)\\
&=
 \int_{H}\langle \Lambda_{t}u, Q_{t}^{-1/2}  z\rangle f\left( z+e^{-At}x \right) N_{0, Q_{t}}\left( dz \right)\\
&\leqslant
\left( P_{t}f^{2}\left( x \right)\right)^{1/2}\left( \int_{H}\vert \langle \Lambda_{t}u, Q_{t}^{-1/2}z \rangle \vert^{2} N_{0, Q_{t}}\left( dz \right)\right)^{1/2}\\
&=
 \left( P_{t}f^{2}\left( x \right)\right)^{1/2}\Vert \Lambda_{t}u\Vert^{2}.
\end{align*}
Note that $\Lambda_{t}$ is  bounded, therefore for bounded measurable functions $f$ we see that $P_{t}f$ is uniformly Lipschitz, and therefore strong Feller.
\end{proof}

\begin{as}\label{OUas}
We now suppose $Q=I$ and that $A$ is diagonal in an orthonormal basis
$\{e_{n}\}_{n=1}^\infty$ of $H$ with eigenvalues $a_{n}$ being a sequence
of positive numbers. Moreover, we assume that $a_{n}/n^p\to \infty$ for some $p>3$.
\end{as}
Note that under this assumption $A^{-1}$ is trace-class for $p>3$, and therefore by Proposition \ref{p.3.1} the operator $Q_{t}$ is trace-class as well. We need the following lemma.

\begin{lem}\label{rb-s3-L31} Suppose $X_t$ is an Ornstein-Uhlenbeck
process with $Q$ and $A$ satisfying Assumption \ref{OUas}. Let $r>q>0$ and $\varepsilon>0$. Then
there exists $t_0$ such that
$$\mathbb{P}^x(\sup_{s\leqslant t_0} |X_s|>r)\leqslant \varepsilon, \qquad x\in B(0,q).$$
\end{lem}

\begin{proof}
We first consider the $n$th component of $X_{s}$. Taking the stopping time $\tau$
identically equal to $t_0$, the main theorem
of  \cite[Theorem 2.5]{GraversenPeskir2000} tells us that

\[
\mathbb{E} \sup_{s\leqslant t_0} |X^n_s|\leqslant
\frac{c\sqrt{\log(1+a_nt_0)}}{\sqrt{a_n}}.
\]
Then by Chebyshev's inequality,
\begin{equation}\label{rb-s3-E31}
\mathbb{P}(\sup_{s\leqslant t_0} |X^n_s|\geqslant d_n)\leqslant
\frac{c\sqrt{\log(1+a_nt_0)}}{d_n\sqrt{a_n}}
\end{equation}
for any positive real number $d_n$.

Choose $\delta>0$ small so that $(p-1)/2>1+\delta$.
Take $d_n=C(r-q) n^{-1/2-\delta}$, where $C$ is chosen so that
$C^2\sum_{n=1}^\infty n^{-1-2\delta}=1$.
Then
$\mathbb{P}(\sup_{s\leqslant t_0} |X^n_s|\geqslant d_n)$
is summable in $n$, and if we choose $n_0$ large enough,
\[
\sum_{n=n_0}^\infty \mathbb{P}(\sup_{s\leqslant t_0} |X^n_s|\geqslant d_n)<\varepsilon/2.
\]
By taking $t_0$ smaller if necessary, we then have

\[
\sum_{n=1}^\infty \mathbb{P}(\sup_{s\leqslant t_0} |X^n_s|\geqslant d_n)<\varepsilon.
\]

Suppose $|x|\leqslant q$ and we start the process at  $x$. By symmetry, we may
assume each coordinate of $x$ is non-negative. Since
\[
|X_s|\leqslant |X_s-x|+|x|,
\]
we observe that in
order for the process to exit the ball $B(0,r)$
before time $t_0$, for some coordinate $n$ we must have $|X_s^n|$
increasing by at least $d_n$. The probability of this happening
is largest when $x_n=0$. But the probability that for some $n$ we have
$|X^n_s|$ increasing by at least $d_n$ in time $t_0$  is bounded by $\varepsilon$.
\end{proof}

\begin{thm}\label{rb-s3-T41} Suppose $X_t$ is an Ornstein-Uhlenbeck
process with $Q$ and $A$ satisfying Assumption \ref{OUas}.
If $h$ is a bounded harmonic function in the ball $B( 0, 1)$, there is a constant $c$  such that
\begin{equation}\label{rb-s3-Emoc}
\vert h\left( x \right) -h\left( y \right)\vert \leqslant c \Vert h \Vert_{0} \vert x-y \vert, \qquad x,y\in B(0,1/2).
\end{equation}
\end{thm}

\begin{proof}
Let $\varepsilon>0$ and let $\tau$ be the exit time from $B(0,1)$.
By Lemma \ref{rb-s3-L31} we  can choose $t_0$ such that
\[
\mathbb{P}^x(\tau<t_0)<\varepsilon, \qquad  x\in B\left(0, 1/2 \right).
\]
If $h$ is harmonic in $B(0,1)$ and $x,y \in B(0,1/2)$,
\[
h(x)=\mathbb{E}^x h(X_\tau)=\mathbb{E}^x[h(X_\tau); \tau<t_0]+
\mathbb{E}^x[h(X_\tau): \tau\geqslant t_0].
\]
The first term is bounded by $\| h\|_0 \varepsilon$.
By the Markov property the second term is
equal to
\[
\mathbb{E}^x[\mathbb{E}^{X_{t_0}}h(X_\tau); \tau\geqslant t_0]=\mathbb{E}^x[h(X_{t_0}); \tau\geqslant t_0],
\]
which differs from $P_{t_0}h(x)$ by at most $\| h\|_0 \varepsilon$.
We have a similar estimate for $h(y)$.
Therefore by Lemma \ref{rb-s3-T51}
\[
|h(x)-h(y)|\leqslant |P_{t_0}h(x)-P_{t_0}h(y)|+ 4\|  h\|_0\varepsilon
\leqslant c(t_0)\vert x-y\vert  \, \| h\|_0 +4\|h\|_0 \varepsilon.
\]
This proves the uniform modulus of continuity.
\end{proof}

\begin{rem}\label{rb-s3-R1} We remark that the constant $c$ in the statement
of  Theorem \ref{rb-s3-T41}
depends on $r$.
Moreover, there does not exist a constant $c$ independent of $z_0$
 such that \eqref{rb-s3-Emoc} holds for $x,y\in B(z_0,r/2)$
when $h$ is harmonic in $B(z_0,r)$.
It is not hard to see that this is the case  even for the two-dimensional Ornstein-Uhlenbeck
process.
\end{rem}

\subsection{Counterexample to the Harnack inequality}

As we have seen, the transition probabilities for the Ornstein-Uhlenbeck process $Z_{t}$ are
\[
P_{t}^{x}\left( dz \right):=N_{e^{-tA}x, Q_{t}}\left( dz \right).
\]
Suppose now that $Q=I$ and $A$ satisfy Assumption \ref{OUas} with $p=1$, but also that $a_n$ is an increasing sequence with $A^{-1}$ being a trace-class operator on $H$. As examples of such $a_n$,  we can take $a_{n}=n^{p}$ for $p>1$.

Denote by $P_{n}$ the orthogonal projection on $H_{n}:=\operatorname{Span}\{ e_{1}, ..., e_{n}\}$. Then
\[
P_{t}^{P_{n}x}\left( dP_{n}z \right):=p_{n}\left( t, P_{n}x, P_{n}z \right) dz,
\]
where
\begin{align*}
 p_{n}( t, &P_{n}x, P_{n}z )\\
&= \prod_{j=1}^{n}\left( \frac{1}{2 \pi} \frac{2a_{j}}{1-e^{-2a_{j}t}}\right)^{1/2} \exp\left(-\frac{2a_{j}\left( z_{j}-e^{-a_{j}t}x_{j}\right)^{2}}{2\left( 1-e^{-2a_{j}t}\right)}\right).
\end{align*}
We would like to consider the Green function $h_{n}$
with pole at $z_n=4e_n$ for $Z_{t}$ killed when $Z_{t}^{1}$ exceeds $6$ in absolute value.
 We use a killed process to insure transience. We will show that
\[
\frac{h_{n}\left( x_n \right)}{h_{n}\left( 0 \right)} \rightarrow \infty
\]
as $n\to \infty$, where $x_n=e_n$.
The key  is to estimate the Green function
\[
h_{n}\left( x, z \right):=\int_{0}^{\infty} \tilde{p}_{n}\left( t, P_{n}x, P_{n}z \right) dt,
\]
where $ \tilde{p}_{n}$ is the density for the killed process. We will prove
 an upper estimate on $h_{n}\left( 0, z_n \right)$ and a lower estimate on $h_{n}\left( x_n, z_n \right)$.

First we  need the following lemma.

\begin{lem}\label{rb-s3-L1} Let $a>0$ and let  $Y_t$ be a one-dimensional Ornstein-Uhlenbeck
process that solves the stochastic differential equation
$$dY_t=\, dB_t-aY_t\, dt,$$
where $B_t$ is a one-dimensional Brownian motion and $a>0$. Let $\widetilde Y$
be $Y$ killed on first exiting $[-6,6]$, let $q(t,x,y)$ be the
transition densities for $Y$, and let $\widetilde q(t,x,y)$ be the
transition densities for $\widetilde Y$.\\
(1)  There exist constants $c$ and $\beta$
such that
 $$\widetilde q(t,0,0)\leqslant ce^{-\beta t}, \qquad t\ge 1.$$
(2) We have $$\frac{\widetilde q(t,0,0)}{q(t,0,0)}\to 1$$
as $t\to 0$.
\end{lem}

\begin{proof} The transition densities of $\widetilde Y$ with respect to
the measure $e^{-x^2/2}\, dx$ are symmetric and by
Mercer's theorem can be written in the form
$$\sum_{i=1}^\infty e^{-\beta_it}\varphi_i(x)\varphi_i(y)$$
with $0<\beta_1\le\beta_2\leqslant \beta_3\leqslant \cdots$.
Here the $\beta_i$ are the eigenvalues and the $\varphi_i$ are
the corresponding eigenfunctions for the Sturm-Liouville problem
$$\left\{\begin{array}{ll}
 Lf(x)&=\frac12 f''(x)-af'(x)=-\beta f(x),\\
f(-6)&=f(6)=0.
\end{array}\right.$$
See \cite[Chapter IV, Section 5]{BassBook1998} for details.
(1) is now immediate.

Let $U$ be the first exit of $Y$ from $[-6,6]$. Using the strong
Markov property at $U$, we have the well known formula
$$q(t,0,0)=\widetilde q(t,0,0)+\int_0^t \mathbb{E}^0
\left[q(t-s,Y_s,0); U\in \, ds\right].$$
Using symmetry, this leads to
\begin{equation}\label{rb-s3-E21}
q(t,0,0)=\widetilde q(t,0,0)+\int_0^t q(t-s,6,0)\mathbb{P}^0(U\in\, ds).
\end{equation}
Now by the explicit formula for $q(r,x,y)$, we see that
$q(t-s,6,0)$ is bounded in $s$ and $t$ and so the second term on the
right hand side of \eqref{rb-s3-E21}  is bounded by a constant times $\mathbb{P}^0(U\leqslant t)$,
which tends to 0 as $t\to 0$. On the other hand, $q(t,0,0)\sim (2\pi t)^{-1/2}
\to \infty$ as $t\to 0$. (2) now follows by dividing both sides of
\eqref{rb-s3-E21} by $q(t,0,0)$.
\end{proof}

We now proceed to an upper estimate for the Green function.

\begin{prop}\label{rb-s3-P22} There are constants $K>0$ and $c>0$ such that
\[
h_{n}\left( 0, z \right) \leqslant K c^{n} a_{n}^{n/2}e^{-16a_{n}}.
\]
\end{prop}

\begin{proof} First for $x=0$ and $z=4e_{n}$ we have
\begin{align*}
 p_{n}( t, P_{n}0&, P_{n}z )=\\
& \prod_{j=1}^{n}\left( \frac{1}{2 \pi} \frac{2a_{j}}{1-e^{-2a_{j}t}}\right)^{1/2} \exp\left(-\frac{16a_{n}}{ 1-e^{-2a_{n}t}}\right).
\end{align*}

\noindent\emph{Step 1.}  Let $t$ be in the interval $0<t\leqslant \frac{1}{2a_{n}}<1$. Then
\[
\prod_{j=1}^{n}\left( \frac{1}{2 \pi} \frac{2a_{j}}{1-e^{-2a_{j}t}}\right)^{1/2} \leqslant \left( \frac{1}{t\pi} \right)^{n/2},
\]
where we used the fact that $a_{n}$ is an increasing sequence.
For  any $t$ we have
\[
\frac{16a_{n}}{ 1-e^{-2a_{n}t}}\geqslant \frac{8}{t};
\]
therefore for $0<t<\frac{1}{2a_{n}}$,
\[
p_{n}\left( t, 0, 4e_{n} \right) \leqslant e^{-8/t} \left( \frac{1}{t\pi} \right)^{n/2}.
\]

The right hand side has its maximum at $\frac{16}{n}$ which is larger than $\frac{1}{2a_{n}}$ for all large enough $n$ by our assumptions on $Q$ and $A$. Thus we can estimate the right hand side by its value at the endpoint $\frac{1}{2a_{n}}$:
\[
p_{n}\left( t, 0, 4e_{n} \right) \leqslant e^{-16a_{n}} \left( \frac{2a_{n}}{\pi} \right)^{n/2}, \qquad 0<t\leqslant \frac{1}{2a_{n}}.
\]
\noindent\emph{Step 2.}  Let $t$ be in the interval $\frac{1}{2a_{n}}<t\leqslant 1$. Denote by $n_{0}$ the index for which $\frac{1}{2a_{n_{0}+1}}<t\leqslant \frac{1}{2a_{n_{0}}}$.
As before
\begin{align*}
 \prod_{j=1}^{n}&\left( \frac{1}{2 \pi} \frac{2a_{j}}{1-e^{-2a_{j}t}}\right)^{1/2} \exp\left(-\frac{16a_{n}}{ 1-e^{-2a_{n}t}}\right)\\
&\leqslant
 \left( \frac{1}{t\pi} \right)^{n_{0}/2}\prod_{j=n_{0}+1}^{n}\left( \frac{1}{2 \pi} \frac{2a_{j}}{1-e^{-2a_{j}t}}\right)^{1/2}\exp\left(-\frac{16a_{n}}{ 1-e^{-2a_{n}t}}\right)\\
&\leqslant
 e^{-16a_{n}} \left( \frac{1}{t\pi} \right)^{n_{0}/2}\prod_{j=n_{0}+1}^{n}\left( \frac{1}{2 \pi} \frac{2a_{j}}{1-e^{-2a_{j}t}}\right)^{1/2}.
\end{align*}
There is  constant $c$ independent of $n$ such that
\[
\frac{1}{2 \pi} \frac{2a_{j}}{1-e^{-2a_{j}t}}\leqslant c a_{j}\leqslant c a_{n}, \qquad j=n_{0}+1, ..., n.
\]
Since  $1/t<2a_{n}$,  there is a constant $c$ such that
\[
\prod_{j=1}^{n}\left( \frac{1}{2 \pi} \frac{2a_{j}}{1-e^{-2a_{j}t}}\right)^{1/2} \exp\left(-\frac{16a_{n}}{ 1-e^{-2a_{n}t}}\right) \leqslant c^{n} a_{n}^{n/2} e^{-16a_{n}}.
\]

\noindent\emph{Step 3.}  For $t>1$ the transition density of the killed process can be estimated by
\[
\prod_{j=2}^{n}\left( \frac{1}{2 \pi} \frac{2a_{j}}{1-e^{-2a_{j}t}}\right)^{1/2} \exp\left(-\frac{16a_{n}}{ 1-e^{-2a_{n}t}}\right) e^{-\beta t}
\]
for some $\beta >0$, using Lemma \ref{rb-s3-L1}(1).
Similarly to Step 2,
\[
\tilde{p}\left( t, 0, 4e_{n} \right) \leqslant c_1^{n-1} a_{n}^{(n-1)/2} e^{-16a_{n}}e^{-\beta t}
\]
for some constant $c_1$. Thus we have that there is a constant $c>0$ such that
\[
\tilde{p}\left( t, 0, 4e_{n} \right) \leqslant
\left\{
\begin{array}{cc}
  c^{n} a_{n}^{n/2} e^{-16a_{n}}, & 0< t < 1,  \\
  c^{n} a_{n}^{n/2} e^{-16a_{n}}e^{-\beta t}, & 1<t.
\end{array}
\right.
\]
Integrating over $t$ from 0 to $\infty$ yields  the result.
\end{proof}

We now obtain the lower bound for the Green function.

\begin{prop}\label{rb-s3-P23} Let $x=e_n$.
There are constants $M>0$, $c>0$ and $\varepsilon >0$ such that
\[
h_{n}\left( x, z \right)\geqslant M c^{n} e^{-16a_{n}} a_{n}^{n/2} \frac{e^{\varepsilon a_{n}}}{a_{n}}.
\]
\end{prop}

\begin{proof} For $x=e_{n}$ and $z=4e_{n}$ we have
 $$p_{n}( t, P_{n}x, P_{n}z )
=
 \prod_{j=1}^{n}\left( \frac{1}{2 \pi} \frac{2a_{j}}{1-e^{-2a_{j}t}}\right)^{1/2} \exp\left(-\frac{a_{n}\left( 4-e^{-a_{n}t}\right)^{2}}{\left( 1-e^{-2a_{n}t}\right)}\right).$$
Observe that
\[
\prod_{j=1}^{n}\left( \frac{1}{2 \pi} \frac{2a_{j}}{1-e^{-2a_{j}t}}\right)^{1/2} \geqslant \left( \frac{1}{2\pi t}\right)^{n/2}.
\]
Consider $t$ in the interval $[ 1/a_{n}, 2/a_{n}]$. When $n$ is large, $2/a_n\leqslant 1$.
Set $v=e^{-a_{n}t}$, so that $v \in [ 1/e^{2}, 1/e ]$ when $t \in [ 1/a_{n}, 2/a_{n}]$. Note that
\[
16-\frac{\left( 4-v \right)^{2}}{1-v^{2}}> 0
\]
for $v \in [ 0, 8/17 ] \supset [ 1/e^{2}, 1/e ]$, so there is a constant $\varepsilon>0$ such that
\[
16-\frac{\left( 4-v \right)^{2}}{1-v^{2}}>\varepsilon, \qquad v \in [ 1/e^{2}, 1/e ].
\]
Thus
\[
 \exp\left(-\frac{a_{n}\left( 4-e^{-a_{n}t}\right)^{2}}{\left( 1-e^{-2a_{n}t}\right)}\right) \geqslant e^{-16a_{n}+\varepsilon a_{n}}.
\]
We now apply Lemma \ref{rb-s3-L1}(2) and obtain
\begin{align*}
 h_{n}\left( x, z \right)&\geqslant \int_{1/a_{n}}^{{2}/{a_{n}}} \tilde p_{n}\left( t, P_{n}x, P_{n}z \right) \,dt \\
&\geqslant
 e^{-16a_{n}+\varepsilon a_{n}} c_2^{n}\int_{{1}/{a_{n}}}^{{2}/{a_{n}}} t^{-n/2}\, dt\\
&=
 e^{-16a_{n}+\varepsilon a_{n}} c_3^{n} a_{n}^{n/2-1} \left(\frac{1-2^{-\frac{n}{2}+1}}{\frac{n}{2}-1} \right).
\end{align*}
Thus we have
\[
h_{n}\left( x, z \right)\geqslant M c^{n} e^{-16a_{n}} a_{n}^{n/2} \frac{e^{\varepsilon a_{n}}}{a_{n}}.
\]
\end{proof}

\begin{thm}\label{rb-s3-T21}
Let $K>0$. There exist functions $h_n$ harmonic and non-negative on $B(0,4)$
and points $x_n$ in $B(0,2)$ such that
$$\frac{h_n(x_n)}{h_n(0)}\ge K$$
for all $n$ sufficiently large.
Thus the  Harnack inequality does not hold for the Ornstein-Uhlenbeck process.
\end{thm}

\begin{proof}
The embedding of the finite dimensional functions $h_n$ into the Hilbert
space framework is done similarly to the proof of Theorem \ref{rb-s2-T1},
but is simpler here as there is no Banach space $W$ to worry about.
We leave the details to the reader.
The theorem then follows by combining Propositions \ref{rb-s3-P22} and \ref{rb-s3-P23}.
\end{proof}

\subsection{Coupling}

It is commonly thought that coupling and the Harnack inequality have
close connections.
Therefore it is interesting that there are infinite-dimensional
Ornstein-Uhlenbeck processes that couple even though they do not
satisfy a Harnack inequality.

We now consider the infinite-dimensional Ornstein-Uhlenbeck defined
as in the previous subsection, but with $a_n=n^p$ and $p=6$.
We have the following theorem.
Given a process $X$, let $\tau_X(r)=\inf\{t: |X_t|\ge r\}$.

\begin{thm}\label{rb-s3-T31} Let $x_0,y_0\in B(0,1)$.
We can construct two infinite-di\-men\-sion\-al
Ornstein-Uhlenbeck processes $X_t$ and $Y_t$ such that $X_0=x_0$ a.s.,
$Y_0=y_0$ a.s., and if $\mathbb{P}^{x_0,y_0}$ is the joint law of the
pair $(X,Y)$, then $$\mathbb{P}^{x_0,y_0}(T_C<\tau_X(2)\land \tau_Y(2))>0,$$
where $T_C=\inf\{t: X_t=Y_t\}$.
\end{thm}

\begin{proof}
Let $W^X_j(t), W^Y_j(t)$, $j=1, 2, \ldots$, all be independent
one-dimensional Brownian motions. Let
$$dX_t^j=dW^X_j(t)-a_j X_t^j\, dt, \qquad X_0^j=x_0^j,$$
and the same for $Y_t^j$, where we replace $dW^X_j$ by $dW^Y_j$ and
$x_0$ by $y_0$. Let
$T^j_C=\inf\{t: X^j(t)=Y^j(t)\}$.
We define
$${\overline Y}^j(t)=\begin{cases} Y^j(t),& t< T^j_C;\\
X^j(t), & t\ge T^j_C.\end{cases}$$

Let $\mathbb{P}^x$ be the law of $X$ when starting at $x$ and similarly for $\mathbb{P}^y$.
Define $\mathbb{P}^{x^j}$ to be the law of $X^j(t)$ started at $x^j$ and so on.
Use Lemma \ref{rb-s3-L31} to choose $t_0$ small such that
$$\sup_{x,y\in B(0,1)}\mathbb{P}^{x,y}(\tau_X(5/4)\land \tau_Y(5/4)\leqslant t_0)\leqslant 1/4.$$
Our first step is to show
\begin{equation}\label{coup-E1}
\sum_{j=1}^\infty \mathbb{P}^{x^j,y^j}(T_C^j>t_0)<\infty.
\end{equation}

The law of $X^j_{t_0/2}$ under $\mathbb{P}^{x^j}$ is that of a normal random variable
with mean $e^{-a_j t_0/2}x^j$ and variance $(1-e^{-a_j t_0/2})/2a_j$.
If $A^X_j$ is the event where $X^j(t_0/2)$ is not in $[-a_j^{-1/4},
a_j^{-1/4}]$, then standard estimates using the Gaussian density
show that  $\sum_j \mathbb{P}^{x_j}(A^X_j)$ is summable.
The same holds if we replace $X$ by $Y$.

Suppose $|x'_j|, |y'_j|\leqslant a_j^{-1/4}$. Let
\begin{equation}\label{rb-sec3.3-E22}
Z^j(t)=(x'_j-y'_j)+(W^X_j(t)-W^Y_j(t))-a_j \int_0^t Z_j(s)\, ds.
\end{equation}
Now $Z^j$ is again a one-dimensional Ornstein-Uhlenbeck process, but with the Brownian
motion replaced by $\sqrt 2$ times a Brownian motion. Using
\eqref{rb-sec3.3-E22} the probability
that $Z_t$ does not hit  0 before time $t_0/2$ is less than or equal to the probability
that $\sqrt 2$ times a Brownian motion does not hit 0 before time $t_0/2$.
This latter probability is less than
or equal to
$$c|x'_j-y'_j|/\sqrt {t_0/2}\leqslant 2ca_j^{-1/4}/\sqrt{t_0/2},$$
which is summable in $j$.

Let $B_j$ be the event $(T^j_C>t_0/2)$. We can therefore conclude that
if $|x'_j|, |y'_j|\leqslant a_j^{-1/4}$, then $\mathbb{P}^{x_j',y_j'}(B_j)$
is summable in $j$.

Now use the Markov property at time $t_0/2$ on the event $(A^j_X)^c\cap (A^j_Y)^c$ to
obtain
\begin{align*}
\mathbb{P}^{x_j,y_j}(T^j_C&>t_0,(A^j_X)^c\cap (A^j_Y)^c)\\
&=\mathbb{E}^{x_j,y_j}\left[\mathbb{P}^{X_j(t_0/2),Y_j(t_0/2)}(T_C^j>t_0/2)
;(A_X^j)^c\cap (A_Y^j)^c\right]\\
&\leqslant \Big(\sup_{|x_j'|,|y_j'|\leqslant a_j^{-1/4}}
\mathbb{P}^{x'_j,y'_j}(T_C^j>t_0/2)\Big) \, \mathbb{P}^{x_j,y_j}((A_X^j)^c\cap (A_Y^j)^c).
\end{align*}
Therefore $$\mathbb{P}^{x_j,y_j}(T^j_C>t_0, (A^j_X)^c\cap (A^j_Y)^c)$$
is summable in $j$.
Since we already know that $\mathbb{P}^{x_j,y_j}(A_X^j)$ and
$\mathbb{P}^{x_j,y_j}(A_Y^j)$ are summable in $j$,
we conclude that \eqref{coup-E1} holds.

Now choose $j_0$ such that
$$\sum_{j=j_0+1}^\infty \mathbb{P}^{x^j,y^j}(T^j_C\ge t_0)<1/4.$$
Choose $\varepsilon$ such that $(1+\varepsilon)^{j_0}\leqslant 5/4$. We will show
that there exists a constant $c_1$ such that for each $j\leqslant j_0$
we have
\begin{equation}\label{coup-E2}
\mathbb{P}^{x^j,y^j} (T^j_C<\tau_X(1+\varepsilon)\land \tau_Y(1+\varepsilon))\ge c_1.
\end{equation}
We know that with probability at least
$1/2$, for each $j>j_0$
each pair $(X^j(t), \overline{Y}^j(t))$ couples before $(X,Y)$ exits
$B(0,5/4)$.
Once we have \eqref{coup-E2}, we know that with probability at least
$c_1$, the pair $(X^j(t), \overline{Y}^j(t))$ couples before exiting $[-1-\varepsilon,1+\varepsilon]$
for $j\leqslant j_0$. Hence, using independence,  with probability at least $c_1^{j_0}$ we have
that for all $j\leqslant j_0$, each pair $(X^j(t), \overline{Y}^j(t))$ couples  before either
$X^j(t)$ or $Y^j(t)$ exits the interval
 $[-1-\varepsilon,1+\varepsilon]$.
 Using the independence again, we have coupling with probability
at least $c_1^{j_0}/2$  of $X$ and $Y$ before either exits the ball of radius $\sqrt 2(5/4)<2$.

To show \eqref{coup-E2}, on the interval $[-1-\varepsilon, 1+\varepsilon]$, the
drift term of the Ornstein-Uhlenbeck  process is bounded, so by using the Girsanov
theorem, it suffices to show with positive probability $W^X_j$ hits
$W^Y_j$ before either exits $[-1-\varepsilon,1+\varepsilon]$. The pair
$(W^X_j(t), W^Y_j(t))$ is a two-dimensional Brownian motion started inside
the square $[-1,1]^2$ and we want to show that it hits the diagonal
$\{y=x\}$ before exiting the square $[-1-\varepsilon,1+\varepsilon]^2$ with positive
probability. This follows from the support theorem for Brownian motion.
See, e.g., \cite[Theorem I.6.6]{BassBook1995}.
\end{proof}

\section{Operators in H\"ormander form}\label{S:hor}

We let $C_b(H)$ denote the set of bounded continuous functions on $H$
with the supremum norm and $C^n_b(H)$ the space of $n$ times continuously
Fr\'echet differentiable functions with all derivatives up to order
$n$ being bounded. $C^{0,1}_b(H)$ will be  the space of all Lipschitz
continuous functions with
$$\|f\|_{0,1}:=\sup_x |f(x)|+\sup_{x\ne y}\frac{|f(x)-f(y)|}{|x-y|}.$$
Finally, $C_b^{1,1}(H)$ will be  the space of Fr\'echet differentiable
functions $f$ wih continuous and bounded derivatives such that $Df$
is Lipschitz continuous; we use the norm
$$\|f\|_{1,1}=\|f\|_{0,1}+\sup_{x\ne y} \frac{|Df(x)-Df(y)|_{H^*}}{|x-y|}.$$

Suppose $H$ is a separable Hilbert space, and $\{ e_{n} \}_{n=1}^{\infty}$ is an orthonormal basis in $H$.
We set
$$(\partial_jf)(x):= (D_{e_j}f)(x).$$

\subsection{Stochastic differential equation}

Let $m\ge 1$ and suppose
$A^1, \ldots, \allowbreak A^m$ are bounded maps from $H$
to $H$. Let $A:=(A^1, \ldots A^m)$.

We assume that
\begin{equation}
a_{i}^{k}\left( x \right):=\langle A^{k}\left( x \right), e_{i}\rangle >0  \text{ for any } x \in H,
\end{equation}
and that we have $a_{i} \in C_{b}^{1, 1}\left( H \right)$ with
\begin{equation}
\Vert A^{k} \Vert_{1, 1}^{2}:=\sum_{i=1}^{\infty} \Vert a_{i}^{k} \Vert_{1, 1}^{2}< \infty.
\end{equation}

 For any $f \in C_{b}^{1}\left( H \right)$ we define
\begin{align*}
& \left( \nabla_{A^{k}}f\right)\left( x \right):=\sum_{i=1}^{\infty} a_{i}^{k}\left( x \right) \left(\partial_{i}f\right)\left( x \right),
\\
& \left( \nabla_{A}f\right)\left( x \right):=\left(\left( \nabla_{A^{1}}f\right)\left( x \right), ..., \left( \nabla_{A^{m}}f\right)\left( x \right)\right).
\end{align*}
Note that
\begin{align*}
\vert \left( \nabla_{A^{k}}f\right)\left( x \right) \vert^{2}& \leqslant \left( \sum_{i=1}^{\infty} \vert a_{i}^{k}\left( x \right)\vert^{2} \right)\left( \sum_{i=1}^{\infty} \vert \left(\partial_{i}f\right)\left( x \right)\vert^{2}  \right)
\\
&\leqslant \Vert A^{k} \vert_{1, 1}^{2} \vert \left(Df\right)\left( x \right)\vert^{2},
\end{align*}
so $\nabla_{A^{k}}f$ and $\nabla_{A}f$ are well-defined for $f \in C_{b}^{1}\left( H \right)$.

Fix a probability space $\left( \Omega, \mathcal{F}, \mathbb{P} \right)$ with a filtration $\mathcal{F}_{t}$, $t \geqslant 0$, satisfying the usual conditions, that is, $\mathcal{F}_{0}$ contains all null sets in $\mathcal{F}$, and $\mathcal{F}_{t}=\mathcal{F}_{t+}=\bigcap_{s>t}\mathcal{F}_{s}$ for all $t \in [0, T]$. Suppose $W_{t}=\left( W_{t}^{1}, ..., W_{t}^{m} \right)$ is a Wiener process on $H ^m$ with covariance operator $Q=\left( Q^{1}, ..., Q^{m} \right)$. We assume that each $Q^{k}, k=1, ..., m$ is a non-negative trace-class operator on $H$ such that
\[
Q^{k}e_{i}=\lambda_{i}^{k}e_{i}, \text{ with } \lambda_{i}^{k}>0 \text{ and } \sum_{i=1}^{\infty}\lambda_{i}^{k}=2,\qquad   k=1, ..., m.
\]

We consider a stochastic differential equation such that the infinitesimal generator of the solution is $L=\sum_{k=1}^{m}\left(\nabla_{A^{k}}\right)^{2}$.

Define $B\left( x \right):=\left( B^{1}\left( x \right), ..., B^{m}\left( x \right)\right), x \in H$ as a linear operator from $H$ to $H ^m$  by
\[
\langle B^{k}\left( x \right)h, e_{i} \rangle :=a_{i}^{k}\left( x \right), \text{ for any } h \in H, \qquad k=1, ..., m,
\]
and $F: H \to H^m$ by
\[
\langle F^{k}\left( x \right), e_{i}\rangle:=\sum_{j=1}^{\infty} a_{j}^{k}\left( x \right)\partial_{j}a_{i}^{k}\left( x \right), \qquad k=1, ..., m.
\]
We can also re-write $B$ and $F$ as
\begin{align*}
& B\left( x \right)\left( h_{1}, ..., h_{m} \right)=A\left( x \right), \text{ for any } \left( h_{1}, ..., h_{m} \right) \in H^m,
\\
& F\left( x\right)=\left( \sum_{i=1}^{\infty}\nabla_{A^{1}}a_{i}^{1}\left( x \right) e_{i}, ..., \sum_{i=1}^{\infty}\nabla_{A^{m}}a_{i}^{m}\left( x \right) e_{i} \right).
\end{align*}

\begin{thm}\label{t.8.2}
\begin{enumerate}

\item Suppose $X_{0}$ is an $H^m$-valued random variable. Then the stochastic differential equation
\[
X_{t}=X_{0}+\int_{0}^{t} B\left( X_{s}\right) \,dW_{s}^{T} +\int_{0}^{t} F\left( X_{s} \right) \,ds,
\]
has a unique solution (up to a.s.\ equivalence) among the processes satisfying
\[
\mathbb{P}\left( \int_{0}^{T} \vert X_{t} \vert_{H^m}^{2} \,dt < \infty\right)=1.
\]
\item If in addition $X_{0}\in L^{2}\left( \Omega, \mathcal{F}_{0}, \mathbb{P}\right)$, then there is a constant $C_{T}>0$ such that
\[
\mathbb{E} \vert X_{t} \vert^{2}\leqslant C_{T}\mathbb{E} \vert X_{0} \vert^{2}.
\]

\item Suppose $f \in C_{b}^{2}\left( H \right)$. Then $v\left( t, x \right):=\mathbb{E}\left(f\left( X_{t}^{x} \right)\right)=P_{t}f\left( x \right)$ is in $C_{b}^{1, 2}\left( H  \right)$ and is the unique solution to the following parabolic equation
\begin{align*}
& \partial_{t}v\left( t, x \right)=Lv ,\qquad t>0, x \in H^m,
\\
& v\left( 0, x \right)=f\left( x \right),
\end{align*}
where $L$ is the operator
\begin{align*}
 \left(Lf\right)\left( x \right)&:=\sum_{k=1}^{m}\left( \nabla_{A^{k}}\nabla_{A^{k}}f\right)\left( x \right)
 \\
&= \sum_{k=1}^{m}\sum_{j=1}^{\infty}a_{j}^{k}\left( x \right)\partial_{j}\left( \sum_{i=1}^{\infty} a_{i}^{k}\left( x \right)\partial_{i}f\left( x \right)\right)
\\
&= \sum_{k=1}^{m}\sum_{i, j=1}^{\infty}a_{i}^{k}a_{j}^{k}\partial^{2}_{ij}f\left( x \right)+\sum_{k=1}^{m}\sum_{i,j=1}^{\infty}a_{j}^{k}\left( x \right)\partial_{j}a_{i}^{k}\left( x \right)\partial_{i}f\left( x \right), \quad x \in H.
\end{align*}
\end{enumerate}

\end{thm}

\begin{proof}
For simplicity of notation we take $m=1$, and write $A^{1}$ for $A$ with corresponding functions $a_{j}$. The proof for the general case is very similar.

In this case  $B\left( x \right), x \in H$,  is a linear operator on $H$ defined by
\[
\langle B\left( x \right)h, e_{i} \rangle :=a_{i}\left( x \right), \text{ for any } h \in H,
\]
and $F: H \to H$ by
\[
\langle F\left( x \right), e_{i}\rangle:=\sum_{j=1}^{\infty} a_{j}\left( x \right)\partial_{j}a_{i}\left( x \right),
\]
or equivalently $B\left( x \right)e_{j}=A\left( x \right)$ , $F\left( x\right)=\sum_{i, j}^{\infty}a_{j}\left( x \right)\partial_{j}a_{i}\left( x \right) e_{i}$.

According to \cite[Theorem 7.4]{DaPratoZabczykBook1}, for this stochastic differential equation to have a unique mild solution it is enough  to check that\\
{(a)} $B\left( x \right) \left( \cdot \right)$ is a measurable map from $H$ to the space $L_{2}^{0}$ of Hilbert-Schmidt operators from $Q^{1/2}H$ to $H$;\\
{(b)} $\Vert B\left( x \right)-B\left( y \right)\Vert_{\operatorname{L_{2}^{0}}}\leqslant C \vert x-y \vert, \quad x, y \in H$;\\
{(c)} $\Vert B\left( x \right)\Vert_{\operatorname{L_{2}^{0}}}^{2}\leqslant K\left( 1+ \vert x\vert^{2}\right), \quad x \in H$;\\
{(d)} $F$ is Lipschitz continuous on $H$ and
$\vert F\left( x \right)\vert\leqslant L\left( 1+ \vert x\vert^{2}\right),\quad  x \in H$.

Let $\{ e_{j}\}_{j=1}^{\infty}$ be an orthonormal basis of $H$. Then $\{ \lambda_{j}^{1/2}e_{j}\}_{j=1}^{\infty}$ is an orthonormal basis of $Q^{1/2}H$. First observe that since $A$ is  bounded we have
\begin{align*}
&
\Vert B\left( x \right) \Vert_{\operatorname{L_{2}^{0}}}^{2}=\sum_{i,j=1}^{\infty} \vert \langle B\left( x \right)\lambda_{j}^{1/2} e_{j}, e_{i}\rangle\vert^{2},
\\
& \vert A\left( x \right)\vert^{2} \sum_{j=1}^{\infty} \lambda_{j} =2\vert A\left( x \right)\vert^{2}\leqslant  C,
\end{align*}
and similarly
\[
\Vert B\left( x \right)-B\left( y \right)\Vert_{L_{2}^{0}}\leqslant \Vert A \Vert_{1, 1} \vert x-y \vert.
\]
The last estimate implies
\[
\Vert B\left( x \right)\Vert_{L_{2}^{0}}\leqslant \operatorname{max}\{C, \vert B\left( 0 \right)\vert \}\left( 1+ \vert x\vert\right)
\]
which proves (a) and (c).
We also have
{\allowdisplaybreaks
\begin{align*}
 \vert F\left( x \right)-F\left( y \right)\vert^{2}&=\sum_{i=1}^{\infty} \langle F\left( x \right)-F\left( y \right), e_{i}\rangle^{2}
\\
&= \sum_{i=1}^{\infty} \left( \sum_{j=1}^{\infty} a_{j}\left( x \right)\partial_{j}a_{i}\left( x \right)-a_{j}\left( y \right)\partial_{j}a_{i}\left( y \right)\right)^{2}
\\
&\leqslant
2\sum_{i=1}^{\infty} \left( \sum_{j=1}^{\infty} \left( a_{j}\left( x \right)-a_{j}\left( y \right)\right)\partial_{j}a_{i}\left( x \right)\right)^{2}\\
&\qquad +
2\sum_{i=1}^{\infty} \left( \sum_{j=1}^{\infty} a_{j}\left( y \right)\left( \partial_{j}a_{i}\left( x \right)-\partial_{j}a_{i}\left( y \right)\right)\right)^{2}\\
&\leqslant
 2 \left( \sum_{j=1}^{\infty}\left( a_{j}\left( x \right)-a_{j}\left( y \right)\right)^{2}\right) \left( \sum_{i, j=1}^{\infty}\left(\partial_{j}a_{i}\left( x \right)\right)^{2}\right)\\
&\qquad +
 2 \left( \sum_{j=1}^{\infty}\left( a_{j}\left( y \right)\right)^{2}\right) \left( \sum_{i, j=1}^{\infty}\left(\partial_{j}a_{i}\left( x \right)-\partial_{j}a_{i}\left( y \right)\right)^{2}\right).
\end{align*}
} 
Now we can use our assumptions on $A$ to see that
\begin{align*}
 &\sum_{j=1}^{\infty}\left( a_{j}\left( x \right)-a_{j}\left( y \right)\right)^{2}\leqslant \sum_{j=1}^{\infty}\Vert a_{i} \Vert_{1, 1}^{2} \vert x-y\vert^{2}=\Vert A \Vert_{1, 1}^{2} \vert x-y\vert^{2},
\\
 &\sum_{j=1}^{\infty}\vert a_{j}\left( y \right)\vert^{2}\leqslant \Vert A \Vert_{1, 1}^{2},
\\
& \sum_{i, j=1}^{\infty}\vert \partial_{j}a_{i}\left( x \right)\vert^{2}=\sum_{i=1}^{\infty} \vert Da_{i}\left( x \right)\vert^{2}\leqslant \Vert A \Vert_{1, 1}^{2},\qquad \mbox{\rm and}
\\
&  \sum_{i, j=1}^{\infty}\left(\partial_{j}a_{i}\left( x \right)-\partial_{j}a_{i}\left( y \right)\right)^{2}=\sum_{i=1}^{\infty} \vert Da_{i}\left( x \right)-Da_{i}\left( y \right)\vert^{2}
\\
&\qquad \leqslant \sum_{i=1}^{\infty} \Vert a_{i} \Vert_{1, 1}^{2} \vert x-y \vert^{2} \leqslant \Vert A \Vert_{1, 1}^{2} \vert x-y \vert^{2},
\end{align*}
which gives Lipschitz continuity for $F$. Finally the estimate for $\vert F\left( x \right)\vert$ follows from the Lipschitz continuity of $F$ together with boundedness of $A$ in a similar fashion to what we did for $B$.

Assertion (2) follows directly from \cite[Theorem 9.1]{DaPratoZabczykBook1}.
Assertion (3) follows from  \cite[Theorem 9.16]{DaPratoZabczykBook1} which says that $P_{t}f$ is the solution to the parabolic type equation with operator
\begin{align*}
 Lv&=\frac{1}{2}\operatorname{tr} v_{xx}\left( B\left( x\right) Q^{1/2}, B\left( x\right) Q^{1/2}\right)+\langle v_{x}, F\left( x \right) \rangle\\
&=
 \frac{1}{2}\sum_{n=1}^{\infty} v_{xx}\left( B\left( x\right) Q^{1/2}e_{n}, B\left( x\right) Q^{1/2}e_{n}\right)+\Big\langle v_{x}, \sum_{i, j}^{\infty}a_{j}\left( x \right)\partial_{j}a_{i}\left( x \right) e_{i} \Big\rangle\\
&= \frac{1}{2}\sum_{n=1}^{\infty} \lambda_{n}v_{xx}\left( \sum_{i=1}^{\infty} a_{i}\left( x\right)e_{i}, \sum_{j=1}^{\infty} a_{j}\left( x\right)e_{j}\right)+ \sum_{i, j}^{\infty}a_{j}\left( x \right)\partial_{j}a_{i}\left( x \right) \langle v_{x}, e_{i} \rangle\\
&=
 \sum_{i, j=1}^{\infty} a_{i}\left( x\right)a_{j}\left( x\right)v_{xx}\left(e_{i},e_{j}\right)+ \sum_{i, j}^{\infty}a_{j}\left( x \right)\partial_{j}a_{i}\left( x \right) \langle v_{x}, e_{i} \rangle\\
&=
 \sum_{i, j=1}^{\infty} a_{i}\left( x\right)a_{j}\left( x\right)\partial^{2}_{ij}v + \sum_{i, j}^{\infty}a_{j}\left( x \right)\partial_{j}a_{i}\left( x \right) \partial_{i} v.
\end{align*}

\end{proof}

\begin{rem} Denote
\[
L^{k}f:=\nabla_{A^{k}}^{2}f=\sum_{i, j=1}^{\infty} a_{i}^{k}\left( x\right)a_{j}^{k}\left( x\right)\partial^{2}_{ij}f + \sum_{i, j}^{\infty}a_{j}^{k}\left( x \right)\partial_{j}a_{i}^{k}\left( x \right) \partial_{i} f,
\]
where $k=1, ..., m$. Suppose $f \in C_{b}^{2}\left( H  \right)$. Then
\begin{align*}
 \left\vert \left(L^{k}f\right)\left( x \right)\right\vert^{2} &\leqslant \sum_{i, j=1}^{\infty}\vert a_{i}^{k}a_{j}^{k}\left( x \right)\vert^{2}\sum_{i, j=1}^{\infty}\vert \partial^{2}_{ij}f\left( x \right)\vert^{2}\\
&\qquad
 +\sum_{j=1}^{\infty}\vert a_{j}^{k}\left( x \right)\vert^{2}\sum_{j=1}^{\infty}\left\vert \sum_{i=1}^{\infty} \partial_{j}a_{i}^{k}\left( x \right) \partial_{i}f\left( x \right)\right\vert^{2} \\
&\leqslant
 \Vert A^{k} \Vert_{1, 1}^{4} \Vert f \Vert_{2}^{2}+\Vert A^{k} \Vert_{1, 1}^{2}
\sum_{i, j=1}^{\infty}\left\vert \partial_{j}a_{i}^{k}\left( x \right)\right\vert^{2}
\sum_{i=1}^{\infty}\left\vert \partial_{i}f\left( x \right)\right\vert^{2} \\
&\leqslant
 2\Vert A^{k} \Vert_{1, 1}^{4} \Vert f \Vert_{2}^{2},
\end{align*}
and therefore $L^{k}$ is well-defined on $C_{b}^{2}\left( H \right)$, and so is $L=\sum_{k=1}^{m} L_{k}$.
\end{rem}

\subsection{Curvature-dimension inequality}

We can write
\[
L=\sum_{k=1}^{m} L_{k}=\sum_{k=1}^{m} \nabla_{A^{k}}^{2}.
\]

For any $f, g \in C_{b}^{2}\left( H \right)$ we define
\begin{align}
& \Gamma\left( f, g \right):=\frac{1}{2}\left( L\left( fg\right)-fL\left( g\right)-gL\left( f\right)\right),\label{rb-sec4.2-E1}
\\
& \Gamma_{2}\left( f \right):=\frac{1}{2}L\left( \Gamma\left( f, f \right)\right)-\Gamma\left( f, Lf \right).\label{rb-sec4.2-E2}
\end{align}

\begin{thm}\label{t.4.5} For any $f, g \in C_{b}^{2}\left( H \right)$,
\begin{align}
& \Gamma\left( f, g \right)=\sum_{k=1}^{m}\left( \nabla_{A^{k}}f \right)\left( \nabla_{A^{k}}g \right),\label{rb-sec4.2-E3}
\\
& \Gamma_{2}\left( f \right)=\sum_{k, l=1}^{m} \Gamma^{(k)}\left( \nabla_{A^{l}}f \right),\label{rb-sec4.2-E4}
\end{align}
where
\[
\Gamma^{(k)}\left( f \right):=\left(  \nabla_{A^{k}}f\right)^{2}.
\]
\end{thm}
\begin{proof}
Note that for functions $f, g \in C_{b}^{2}\left( H \right)$
\begin{align}
 L_{k}\left( fg\right)&=fL_{k}\left( g\right)+gL_{k}\left( f\right)+2\left(\sum_{i}a_{i}^{k}\partial_{i}f\right)\left(\sum_{j}a_{j}^{k}\partial_{j}g\right) \label{e.4.3}
\\
&= fL_{k}\left( g\right)+gL_{k}\left( f\right)+2\left(\nabla_{A^{k}}f \right)\left(\nabla_{A^{k}}g \right),\notag
\end{align}
and therefore
\begin{equation}\label{e.9.1}
L\left( fg\right)=fL\left( g\right)+gL\left( f\right)+2\sum_{k=1}^{m} \left(\nabla_{A^{k}}f \right)\left(\nabla_{A^{k}}g \right).
\end{equation}
Hence
$$ \Gamma\left( f, g \right)=\frac{1}{2}\left( L\left( fg\right)-fL\left( g\right)-gL\left( f\right)\right)=\sum_{k=1}^{m}\left( \nabla_{A^{k}}f \right)\left( \nabla_{A^{k}}g \right),$$
and in particular
$ \Gamma\left( f \right):=\Gamma\left( f, f \right)=\sum_{k=1}^{m}\left( \nabla_{A^{k}}f \right)^{2}.$
Before we find $\Gamma_{2}\left( f \right)$ we need the following calculation.
\begin{align}
 [L_k,\partial_i]&:=\left( L_{k}\partial_{i}-\partial_{i}L_{k}\right)f=
 \sum_{jm}\left( a_{j}^{k}\partial_{j}a_{m}\right)\partial^{2}_{im}f+\sum_{jm}a_{j}^{k}a_{m}^{k}\partial^{3}_{ijm}f \notag\\
&\qquad -\partial_{i}\left(\sum_{jm} a_{j}^{k}\partial_{j}a_{m}^{k} \partial_{m}f+\sum_{jm}a_{j}^{k}a_{m}^{k}\partial^{2}_{jm}f\right) \label{e.9.2}
\\
&= -\sum_{jm}\left( \partial_{i}a_{j}^{k}\partial_{j}a_{m}^{k}+a_{j}^{k}\partial^{2}_{ij}a_{m}^{k}\right)\partial_{m}f -2\sum_{jm}\left(a_{m}^{k}\partial_{i}a_{j}^{k}\right) \partial^{2}_{jm}f. \notag
\end{align}
Use \eqref{e.9.2} to see that
{\allowdisplaybreaks
\begin{align}
 &\sum_{i} a_{i}^{l}\left([ L_{k}, \partial_{i}]f\right) \notag
\\
&= -\sum_{m}\left(\sum_{ij}\left( a_{i}^{l}\partial_{i}a_{j}^{l}\partial_{j}a_{m}^{l}+ a_{i}^{l}a_{j}^{l}\partial^{2}_{ij}a_{m}^{l}\right)\right)\partial_{m}f
\notag\\
&\qquad
-2\sum_{ijm}\left(a_{i}^{l}a_{m}^{l}\partial_{i}a_{j}^{l}\right) \partial^{2}_{jm}f \notag
\\
&= -\sum_{m}\left( L_{k}a_{m}^{l}\right)\partial_{m}f-2\sum_{ijm}\left(a_{i}^{l}a_{m}^{l}\partial_{i}a_{j}^{l}\right) \partial^{2}_{jm}f \notag
\\
&= -\sum_{m}\left( L_{k}a_{m}^{l}\right)\partial_{m}f-2\sum_{j}\left(\sum_{i}a_{i}^{l}\partial_{i}a_{j}^{l}\right)\left( \sum_{m} a_{m}^{l}\partial^{2}_{m j}f\right) \notag
\\
&=
-\sum_{m}\left( L_{k}a_{m}^{l}\right)\partial_{m}f-2\sum_{j}\left(\nabla_{A^{l}} a_{j}^{l}\right)\left(\nabla_{A^{l}}\partial_{j}f\right).\label{e.4.5}
\end{align}
}

Now we can deal with $\Gamma_{2}\left( f \right)$.
 We use \eqref{e.9.1} in the first line.
\begin{align*}
 \frac{1}{2}L\left( \Gamma\left( f\right)\right)&=\frac{1}{2}\sum_{k=1}^{m} L_{k}\left( \Gamma\left( f \right)\right)=\frac{1}{2}\sum_{k=1}^{m}L_{k}\left( \sum_{l=1}^{m}\left( \nabla_{A^{l}}f \right)^{2} \right)
\\
&= \sum_{k, l=1}^{m} \left( \left( \nabla_{A^{l}}f\right) \left( L_{k} \nabla_{A^{l}}f\right)+\Gamma^{(k)}\left( \nabla_{A^{l}}f \right)\right).
\end{align*}

The second term in $\Gamma_{2}\left( f \right)$ is
\[
\Gamma\left( f, Lf \right)=\sum_{l=1}^{m} \left( \nabla_{A^{l}}f\right)\left( \nabla_{A^{l}}Lf\right)=\sum_{k, l=1}^{m} \left( \nabla_{A^{l}}f\right)\left( \nabla_{A^{l}}L_{k}f\right).
\]
Thus
\[
\Gamma_{2}\left( f \right) = \sum_{k, l=1}^{m}  \left( \nabla_{A^{l}}f\right) \left( [ L_{k}, \nabla_{A^{l}}]f \right)+\sum_{k, l=1}^{m} \Gamma_{k}\left( \nabla_{A^{l}}f \right).
\]
By \eqref{e.4.3} we have
{\allowdisplaybreaks
\begin{align*}
  [ L_{k}, \nabla_{A^{l}}]f&= L_{k}\left( \sum_{j=1}^{\infty} a_{j}^{l} \partial_{j}f \right)- \sum_{j=1}^{\infty} a_{j}^{l} \partial_{j}L_{k}f
\\
&= \sum_{j=1}^{\infty} L_{k}\left(a_{j}^{l}\right) \partial_{j}f+ \sum_{j=1}^{\infty} a_{j}^{l} L_{k}\partial_{j}f +2 \sum_{j=1}^{\infty}\left( \nabla_{A^{k}}a_{j}^{l}\right)\left( \nabla_{A^{k}}\partial_{j}f\right) \\
&\qquad - \sum_{j=1}^{\infty} a_{j}^{l} \partial_{j}L_{k}f
\\
&= \sum_{j=1}^{\infty} L_{k}\left(a_{j}^{l}\right) \partial_{j}f+ \sum_{j=1}^{\infty} a_{j}^{l} [ L_{k}, \partial_{j}]f +2 \sum_{j=1}^{\infty}\left( \nabla_{A^{k}}a_{j}^{l}\right)\left( \nabla_{A^{k}}\partial_{j}f\right).
\end{align*}
}
We can use \eqref{e.4.5} to see that
$[ L_{k}, \nabla_{A^{l}}]f=0$ for $ k, l=1, ..., m.$
Thus \eqref{rb-sec4.2-E4} holds.
\end{proof}

\begin{cor} $L$ satisfies the curvature-dimension inequality $\operatorname{CD}\left( 0, m \right)$
\begin{equation}\label{e.CD}
\Gamma_{2}\left( f \right) \geqslant \frac{1}{m}\left( Lf \right)^{2}.
\end{equation}
Moreover, for $m=1$ we have
$\Gamma_{2}\left( f \right)=\left( Lf \right)^{2}$.
\end{cor}

\begin{proof}
Note that by the Cauchy–-Schwarz inequality
\[
\sum_{k, l=1}^{m} \Gamma_{k}\left( \nabla_{A^{l}}f \right)=\sum_{k, l=1}^{m} \left( \nabla_{A^{k}} \nabla_{A^{l}}f \right)^{2}\geqslant \frac{1}{m}\left( \sum_{k=1}^{m} \nabla_{A^{k}}^{2} f \right)^{2}=\frac{1}{m}\left( Lf \right)^{2}.
\]
Therefore
\[
\Gamma_{2}\left( f \right) \geqslant \sum_{k, l=1}^{m}  \left( \nabla_{A^{l}}f\right) \left( [ L_{k}, \nabla_{A^{l}}]f \right)+\frac{1}{m}\left( Lf \right)^{2}.
\]
\end{proof}

We need chain rules for the operators $\Gamma$ and $\Gamma_{2}$.

\begin{prop}
Let $\Psi$ be a $C^{\infty}$ function on $\mathbb{R}$ and suppose $f$ is in the domain of $L$. Then
\begin{equation}\label{e.9.3}
L\Psi\left( f \right)=\Psi^{\prime}\left( f \right)Lf+\Psi^{\prime \prime}\left( f \right)\Gamma\left( f, f \right),
\end{equation}
\begin{equation}\label{e.9.5}
\Gamma\left( \Psi\left( f \right), g\right)=\Psi^{\prime}\left( f \right)\Gamma\left( f, g \right),
\end{equation}
\begin{align}
\Gamma_{2}\left( \Psi\left( f \right)\right)&=\left( \Psi^{\prime \prime}\left( f \right)\right)^{2}\left(\Gamma\left( f\right)\right)^{2}+\left( \Psi^{\prime }\left( f \right)\right)^{2}\Gamma_{2}\left( f\right)
\label{e.9.6} \\
&\qquad +\Psi^{\prime }\left( f \right) \Psi^{\prime \prime}\left( f \right)\Gamma\left( f, \Gamma\left( f\right) \right).\notag
\end{align}
\end{prop}

\begin{proof}
Suppose $\Psi \in C^{\infty}\left( \mathbb{R}\right)$. Recall that we can write $L$ as
 $Lf=\sum_{k=1}^{m} L_{k}=\sum_{k=1}^{m}\nabla_{A^{k}}^{2}f,$
 where
 $\nabla_{A^{k}}f:=\sum_{i=1}^{\infty}a_{i}^{k}\partial_{i}f.  $
 It is clear that
 \begin{equation}\label{e.9.4}
  \nabla_{A^{k}}\left( \Psi\left( f \right) \right)=\Psi^{\prime}\left( f \right)\nabla_{A^{k}}f.
 \end{equation}
Then
 \begin{align*}
 \nabla_{A^{k}}\nabla_{A^{k}}\left( \Psi\left( f \right) \right)&=\nabla_{A^{k}}\left( \Psi^{\prime}\left( f \right)\right)\nabla_{A^{k}}f+ \Psi^{\prime}\left( f \right)\nabla_{A^{k}}\left(\nabla_{A^{k}}f\right)
\\
&= \Psi^{\prime \prime}\left( f \right)\left(\nabla_{A^{k}}f \right)^{2}+ \Psi^{\prime}\left( f \right)\nabla_{A^{k}}\left(\nabla_{A^{k}}f\right)
\\
&= \Psi^{\prime}\left( f \right)L_{k}f+\Psi^{\prime \prime}\left( f \right)\Gamma_{k}\left( f\right),
 \end{align*}
 which implies \eqref{e.9.3} by Theorem \ref{t.4.5}.

Now we can easily show \eqref{e.9.5}. Indeed, using \eqref{e.9.4} we have
\begin{align*}
 \Gamma_{k}\left( \Psi\left( f \right), g\right)&=\left( \nabla_{A^{k}} \Psi\left( f \right) \right)\left( \nabla_{A^{k}} g \right)
\\
&= \Psi^{\prime}\left( f \right) \left( \nabla_{A^{k}}  f  \right)\left( \nabla_{A^{k}} g \right)=\Psi^{\prime}\left( f \right)\Gamma_{k}\left( f, g \right).
\end{align*}
In particular, \eqref{e.9.5} implies
\[
\Gamma\left( \Psi\left( f \right)\right)=\left( \Psi^{\prime}\left( f \right)\right)^{2}\Gamma\left( f\right).
\]
Now we would like to prove \eqref{e.9.6}. First, using \eqref{e.9.5} twice we see that
\begin{equation}\label{e.8.8}
\Gamma\left( \Psi\left( f \right)\right)=\left(\Psi^{\prime}\left( f \right)\right)^{2}\Gamma\left( f \right).
\end{equation}
By \eqref{e.9.1} and \eqref{e.9.3}
\begin{align*}
 \frac{1}{2}&L\Gamma\left( \Psi\left( f \right)\right)
\\
&= \frac{1}{2}\Gamma\left(  f \right)L\left(\left(\Psi^{\prime}\left( f \right)\right)^{2}\right)+\frac{1}{2}\left(\Psi^{\prime}\left( f \right)\right)^{2}L\Gamma\left(  f \right)+\Gamma\left( \left(\Psi^{\prime}\left( f \right)\right)^{2}, \Gamma\left(  f \right)\right)
\\
&= \Psi^{\prime}\left( f \right)\Psi^{\prime \prime}\left( f \right)\left( Lf \right)\Gamma\left(  f \right)+\left( \left( \Psi^{\prime \prime}\left( f \right) \right)^{2}+\Psi^{\prime}\left( f \right)\Psi^{\prime \prime \prime}\left( f \right) \right)\left(\Gamma\left( f\right)\right)^{2}
\\
&\qquad +\frac{1}{2}\left( \Psi^{\prime }\left( f \right) \right)^{2} L\Gamma\left( f \right) +2\Psi^{\prime}\left( f \right)\Psi^{\prime \prime }\left( f \right)\Gamma\left( f, \Gamma\left( f \right)\right).
\end{align*}
Now use \eqref{e.9.1} and \eqref{e.9.6} repeatedly
to obtain
\begin{align*}
 \Gamma\left( \Psi\left( f \right), L\Psi\left( f \right) \right)&=
 \Gamma\left( \Psi\left( f \right),\Psi^{\prime}\left( f \right) Lf\right)+\Gamma\left( \Psi\left( f \right), \Psi^{\prime \prime}\left( f \right) \Gamma\left( f \right)\right)
\\
&= \left(\Psi^{\prime}\left( f \right)\right)^{2}\Gamma\left( f, Lf\right)+\Psi^{\prime}\left( f \right)\Psi^{\prime \prime}\left( f \right)\left( Lf\right)\Gamma\left( f\right)
\\
&\qquad + \Psi^{\prime}\left( f \right)\Psi^{\prime \prime}\left( f \right)\Gamma\left( f, \Gamma\left( f \right)\right)+\Psi^{\prime}\left( f \right)\Psi^{\prime \prime\prime}\left( f \right)\left(\Gamma\left( f \right)\right)^{2}.
\end{align*}
Note that we also used the fact that
\[
\Gamma\left( f, gh \right)=g\Gamma\left( f, h \right)+h\Gamma\left( f, h \right).
\]
Combining these two calculations gives \eqref{e.9.6}.
\end{proof}

\begin{cor}
By \eqref{e.9.6} with $\Psi\left( x \right)=\log x, x>0$, and $g>0$ we see that
\begin{equation}\label{e.9.7}
\Gamma_{2}\left( \log g \right)=\frac{\left(\Gamma\left( g \right)\right)^{2}}{g^{4}}-\frac{\Gamma\left( g, \Gamma\left( g \right)\right)}{g^{3}}+\frac{\Gamma_{2}\left( g\right)}{g^{2}}.
\end{equation}
\end{cor}

\subsection{Li-Yau estimate}

The following is the Li-Yau estimate in our context. In this proof we follow an argument in \cite{BakryLedoux2006}, which they used to prove a finite-dimensional logarithmic Sobolev inequality for heat kernel measures. 

\begin{thm}
\begin{equation}\label{e.9.11}
L\left( \log P_{t}f \right)>-\frac{1}{2t}.
\end{equation}
\end{thm}
\begin{proof}
By \eqref{e.9.5} with $\Psi\left( x \right)=\log x, x>0$, $f>0$, and $0\leqslant s \leqslant t$,
\[
\Gamma\left( P_{t-s}f \right):=\Gamma\left( P_{t-s}f, P_{t-s}f \right)=\left( P_{t-s}f \right)^{2} \Gamma\left( \log P_{t-s}f \right)
\]
Define for $f>0$
\[
\varphi\left( s \right):=P_{s}\left( P_{t-s}f \Gamma\left( \log P_{t-s}f \right)\right)=P_{s}\left( \frac{\Gamma\left( P_{t-s}f \right)}{P_{t-s}f}\right).
\]
Then with $g:=P_{t-s}f$ and $\partial_{s}g=-Lg$ we see that by \eqref{e.9.3} and \eqref{e.9.5}
\begin{align*}
 &\varphi^{\prime}( s )=\partial_{s}\Big(P_{s}\Big( \frac{\Gamma( g )}{g}\Big)\Big)
\\
&= P_{s}\Big( L\Big( \frac{\Gamma( g )}{g}\Big)-\frac{2\Gamma(g, Lg )}{g}+
\frac{\Gamma( g) Lg}{g^{2}}\Big)
\\
&= P_{s}\Big(
\Big({L\Gamma( g )}{g}
+\Gamma( g )L\Big( \frac{1}{g}\Big)
+2\Gamma\Big( \Gamma(g), \frac{1}{g}\Big)
-\frac{2\Gamma(g, Lg )}{g}+
\frac{\Gamma( g) Lg}{g^{2}}\Big)
\\
&= P_{s}\Big(
\Gamma( g )\Big( \frac{2\Gamma( g )}{g^{3}}-\frac{Lg}{g^{2}}\Big)
-\frac{2\Gamma( \Gamma(g), g)}{g^{2}}+\frac{L\Gamma( g )
 -2\Gamma(g, Lg )}{g}+
\frac{\Gamma( g) Lg}{g^{2}}\Big)
\\
&=2P_{s}\Big(\frac{(\Gamma( g ))^{2}}{g^{3}}-\frac{\Gamma( g, \Gamma( g )\Big)}{g^{2}}+\frac{\Gamma_{2}( g)}{g}\Big)=
2P_{s}\Big(g \Gamma_{2}( \log g )\Big)
\end{align*}
by \eqref{e.9.7}. We use the curvature-dimension inequality \eqref{e.CD} to obtain
\begin{equation}\label{e.8.10}
\varphi^{\prime}\left( s \right)\geqslant \frac{2}{m} P_{s}\left( g \left( L \log g \right)^{2} \right).
\end{equation}
In particular, this means that $\varphi$ is non-decreasing, and therefore
\[
\varphi\left( 0 \right)= P_{t}f \Gamma\left( \log P_{t}f \right)\leqslant P_{t}\left( f \Gamma\left( \log f \right)\right)=\varphi\left( t \right).
\]
Using the chain rule \eqref{e.9.5} we get
\[
P_{t}f \Gamma\left( \log P_{t}f \right)=\frac{\Gamma\left( P_{t}f \right)}{P_{t}f}\leqslant P_{t}\left( \frac{\Gamma\left( f \right)}{f}\right) =P_{t}\left( f \Gamma\left( \log f \right)\right).
\]
This inequality together with \eqref{e.9.3} gives
\[
P_{t}fL\left( \log P_{t}f\right)=LP_{t}f-\frac{\Gamma\left( P_{t}f \right)}{P_{t}f}\geqslant LP_{t}f- P_{t}\left( \frac{\Gamma\left( f \right)}{f}\right)=P_{t}\left( f L\left( \log f \right)\right).
\]
Thus
\begin{equation}\label{e.9.9}
P_{t}fL\left( \log P_{t}f\right)\geqslant P_{t}\left( f L\left( \log f \right)\right).
\end{equation}
We need more information about $\varphi$ to complete the proof. Our expression for $\varphi^{\prime}$ can be rewritten using the chain rule \eqref{e.9.3}
as
\[
\varphi^{\prime}\left( s \right)= P_{s}\left( g \left( L \log g \right)^{2} \right)=P_{s}\left( \frac{1}{g} \left( Lg-\frac{\Gamma\left( g \right)}{g} \right)^{2} \right).
\]
Note that since $g>0$ we have
\begin{align*}
 P_{s}\left( Lg-\frac{\Gamma\left( g \right)}{g} \right)&=P_{s}\left( \sqrt{g}\left( \frac{1}{\sqrt{g}}\left( Lg-\frac{\Gamma\left( g \right)}{g} \right)\right)\right)
\\
&\leqslant \left( P_{s}g\right)^{1/2}\left( P_{s}\left( \frac{1}{g}\left( Lg-\frac{\Gamma\left( g \right)}{g} \right)^{2}\right)\right)^{1/2},
\end{align*}
so
\[
P_{s}\left( \frac{1}{g}\left( Lg-\frac{\Gamma\left( g \right)}{g} \right)^{2}\right)\geqslant \frac{\left(P_{s}\left( Lg-\frac{\Gamma\left( g \right)}{g} \right) \right)^{2}}{P_{s}g}
\]
Since $\varphi\left( s \right)=P_{s}\left( \frac{\Gamma\left( g \right)}{g} \right)$, the last estimate becomes
\[
\varphi^{\prime}\left( s \right)\geqslant 2\frac{\left( P_{s}Lg-\varphi\left( s\right) \right)^{2}}{P_{s}g}.
\]
Now use the definition of  $g$ and the fact that $L$ and $P_{s}$ commute to see that $P_{s}g=P_{t}f$, so we have that for $0\leqslant s \leqslant t$
\[
\varphi^{\prime}\left( s \right)\geqslant 2\frac{\left(L P_{t}f-\varphi\left( s\right) \right)^{2}}{P_{t}f}=2\frac{\left(\varphi\left( s\right)- L P_{t}f\right)^{2}}{P_{t}f}.
\]
Thus for all $s$ such that $\varphi^{\prime}\left( s \right)>0$ we have
\[
-\partial_{s}\left( \frac{1}{\varphi\left( s\right)- L P_{t}f}\right)\geqslant\frac{2}{P_{t}f}>0.
\]
By \eqref{e.8.10} we know that $\varphi^{\prime}\left( s \right)\geqslant 0$, and by integrating this estimate from $0$ to $t$, we obtain
\[
\frac{1}{\varphi\left( 0\right)- L P_{t}f}-\frac{1}{\varphi\left( t\right)- L P_{t}f}\geqslant\frac{2t}{P_{t}f}.
\]
That is,
\[
\frac{\varphi\left( t\right)-\varphi\left( 0\right)}{\left(\varphi\left( 0\right)- L P_{t}f\right)\left( \varphi\left( t\right)- L P_{t}f\right)}\geqslant\frac{2t}{P_{t}f}>0.
\]
Since $\varphi$ is non-decreasing, the numerator on the left is non-negative. Since the right hand side of the estimate is positive, no matter what the sign of the denominator on the left, the following estimate holds:
\[
\varphi\left( t\right)-\varphi\left( 0\right)\geqslant\frac{2t}{P_{t}f}\left(\varphi\left( 0\right)- L P_{t}f\right)\left( \varphi\left( t\right)- L P_{t}f\right).
\]
Similarly to the proof of \eqref{e.9.9}
\begin{align*}
& \varphi\left( 0 \right)-LP_{t}f=\frac{\Gamma\left( P_{t}f\right)}{P_{t}f}-LP_{t}f=-P_{t}fL\left( \log P_{t}f \right),
\\
& \varphi\left( t \right)-LP_{t}f=P_{t}\left(\frac{\Gamma\left(f\right)}{f}\right)-LP_{t}f=-P_{t}\left(fL\left( \log f \right)\right).
\end{align*}
Finally we have
\begin{equation}\label{e.9.10}
P_{t}fL\left( \log P_{t}f \right)\geqslant P_{t}\left(fL\left( \log f \right)\right)\left(1+2t L\left( \log P_{t}f \right)\right).
\end{equation}

Now we are ready to prove \eqref{e.9.11}. We only need to check \eqref{e.9.11} when $L\left( \log P_{t}f \right)<0$. In this case, by \eqref{e.9.9}
\[
P_{t}\left( f L\left( \log f \right)\right)<0,
\]
and therefore \eqref{e.9.10} implies
\[
1+2t L\left( \log P_{t}f \right)>0.
\]
\end{proof}
\begin{cor}\label{cor.8.8} For $f>0$
\[
-\partial_{t}\left( \log P_{t}f \right)<\frac{1}{2t}-\Gamma\left( \log P_{t}f \right).
\]
\end{cor}
\begin{proof}
By \eqref{e.9.3} and \eqref{e.8.8}
\begin{align*}
 L\left( \log P_{t}f \right)&=\frac{LP_{t}f}{P_{t}f}-\frac{\Gamma\left( P_{t}f \right)}{\left( P_{t}f \right)^{2}}
\\
&= \frac{\partial_{t}P_{t}f}{P_{t}f}-\Gamma\left( \log P_{t}f \right)=\partial_{t}\left( \log P_{t}f \right)-\Gamma\left( \log P_{t}f \right)>-\frac{1}{2t}.
\end{align*}
\end{proof}

\subsection{Distances}\label{ss.4.4} For the purposes of the next subsection we need to introduce several distances related to the gradient $\nabla_{A}$. A natural distance as described in \cite{Bakry2006Tata} is:\\
\[
d\left( x, y \right):= \sup_{\{ f: \Gamma \left( f \right)\leqslant 1 \}} \left( f\left( y \right)-f\left( x \right) \right), \qquad x,y\in H.
\]

We will need another distance which is better suited for the proof of the parabolic Harnack inequality, and it will turn out that this distance is equal to the one we have just defined. First we note that for any $x\in H$ there is a smooth path $\gamma_{A}: [0, \infty) \to H^m$ (possibly defined only on a finite subinterval $[ 0, T ]$ of $\mathbb{R}_{+}$) such that
\begin{equation}\label{e.6.16}
\dot{\gamma_{A}}\left( t \right)=A\left(\gamma_{A}\left( t \right)\right),\qquad  \gamma_{A}\left( 0 \right)=x.
\end{equation}
This is equivalent to solving a system of ordinary differential equations, which gives $\gamma_{A}$ implicitly as the solution to
\[
x_{j}+\int \frac{d \gamma_{j}}{a_{j}\left( \gamma \right)}=t.
\]
Using the assumption that $a_{j}>0$, we can determine $\gamma_{A}$ as a function of $t$.

An admissible component of $x$ is defined as
\[
V_{A}\left( x \right):=\{ \gamma_{A}\left( s\right), \text{ where } s\in [0, T ],  \dot{\gamma_{A}}\left( s \right)=A\left(\gamma_{A}\left( s \right)\right), \gamma_{A}\left( 0 \right)=x \}
\]
as described by \eqref{e.6.16}.

\begin{ex}{\rm Suppose $a_{j}\left( x \right)=c_{j}$. Then $\gamma$ is a straight line, and so $V_{A}$ is a straight line through $x$ in the direction of $\left( c_{1}, c_{2}, .... \right)$. In particular, if $H=\mathbb{R}^{2}$, and  $a_{1}\left( x \right)=1$ and $a_{2}\left( x \right)=0$, then  $V_{A}$ is a horizontal line through $x$.
}\end{ex}

\begin{df}\label{d.6.12} Let $x \in H$, and define
\[
d_{arc}\left( x, y \right):=\left\{
                                     \begin{array}{ll}
T_{y}, & y \in V_{A}\left( x \right); \\
& \\
+\infty, & y \notin V_{A}\left( x \right),
                                     \end{array}
                                   \right.
\]
where the path $\gamma_{A}$ is described by \eqref{e.6.16} with $\gamma_{A}\left( T_{y} \right)=y$.
\end{df}
\begin{rem} Note that our assumptions on $A$ are essential for the definition of the distance function $d_{arc}$ as we use the ordinary differential equations \eqref{e.6.16} to find $\gamma_{A}$.
\end{rem}
\begin{thm} For any $x, y \in H$
\[
d\left( x, y \right)=d_{arc}\left( x, y \right).
\]
\end{thm}

\begin{proof} Fix $x \in H$. We will consider the case when  $d_{arc}\left( x, y \right)=\infty$ or $d\left( x, y \right)=\infty$ later, so for now we assume that both distances are finite.

Let $\gamma$ be any path connecting $x$ and $y$ with $\gamma\left( s \right)=y$. Note that since $d_{arc}\left( x, y \right)<\infty$, we have $y \in V_{A}\left( x \right)$. Then
\begin{equation}
 d\left( x, y \right)=\sup_{\{ f: \Gamma \left( f \right)\leqslant 1 \}} \left( f\left( y \right)-f\left( x \right) \right)
= \sup_{\{ f: \Gamma \left( f \right)\leqslant 1 \}} \int\limits_{0}^{s} \langle \nabla f|_{\gamma\left( t \right)}, \dot{\gamma}\left( t \right) \rangle \,dt \label{e.6.17}
\end{equation}

Choosing $f_{A}$ such that $\nabla f_{A}=\frac{A}{\vert A \vert^{2}}$, then
\[
\Gamma \left( f_{A} \right)=\vert \nabla_{A} f_{A} \vert^{2}=\langle \nabla f_{A}, A \rangle^{2}=1,
\]
and therefore for the function $f_{A}$
\[
d\left( x, y \right)\geqslant f_{A}\left( y \right)-f_{A}\left( x \right)=\int\limits_{0}^{T_{y}} \langle \nabla f_{A}, \dot{\gamma_{A}}\left( t \right) \rangle \,dt=\int\limits_{0}^{T_{y}}1 \,dt=T_{y}=d_{arc}\left( x, y \right).
\]
Again, by \eqref{e.6.17},
\begin{align}
 d\left( x, y \right)&=\sup_{\{ f: \Gamma \left( f \right)\leqslant 1 \}} \int\limits_{0}^{T_{y}} \langle \nabla f|_{\gamma_{A}\left( t \right)}, \dot{\gamma_{A}}\left( t \right) \rangle \,dt\notag
\\
&= \sup_{\{ f: \Gamma \left( f \right)\leqslant 1 \}} \int\limits_{0}^{T_{y}} \langle \nabla f|_{\gamma_{A}\left( t \right)}, \gamma_{A}\left(\gamma\left( t \right)\right) \rangle \,dt \label{e.6.18}
\\
&= \sup_{\{ f: \Gamma \left( f \right)\leqslant 1 \}} \int\limits_{0}^{T_{y}} \nabla_{A} f|_{\gamma_{A}\left( t \right)} \,dt \leqslant \int\limits_{0}^{T_{y}} 1 \,dt=d_{arc}\left( x, y \right). \notag
\end{align}
Finally we want to show that both distances are infinite for the same $y$.  Define a function
\[
f_{N}\left( z \right):=\left\{
                                     \begin{array}{ll}
0, & z \in V_{A}\left( x \right); \\
& \\
N, & z \notin V_{A}\left( x \right)
                                     \end{array}
                                   \right.
\]
for some $N$. Note that $\Gamma\left( f_{N} \right)=0$. Suppose $d_{arc}\left( x, y \right)=\infty$, so $f_{N}\left( y \right)=N$. Then
\[
d\left( x, y \right)\geqslant f_{N}\left( y \right)-f_{N}\left( x \right)=N.
\]
By taking $N\to \infty$ we see that $d\left( x, y \right)=+\infty$.

Next suppose that $d\left( x, y \right)=\infty$. Then there are functions $f_{N}$ with $\Gamma\left( f_{N} \right)\leqslant 1$ such that
$ f_{N}\left( y \right)-f_{N}\left( x \right)\rightarrow +\infty$
as $N\to \infty$.
Similarly to \eqref{e.6.18} (if we assume that $d_{arc}\left( x, y \right)<\infty$ to find $\gamma_{A}$) we see that
\[
+\infty=\lim_{N\to \infty}f_{N}\left( y \right)-f_{N}\left( x \right)\leqslant T_{y}=d_{arc}\left( x, y \right),
\]
and therefore $d_{arc}\left( x, y \right)=+\infty$.
\end{proof}

\subsection{The parabolic Harnack inequality}

\begin{thm}\label{t.4.13} Suppose $u$ is a positive solution to the heat equation
\[
\partial_{t}u=Lu, \qquad  u(0,\cdot)=f.
\]
Then for any $0\leqslant t_{1} < t_{2} \leqslant 1$ and $x, y$ in the same admissible component, say, $V_{A}\left( x \right)$, we have
\[
\log u\left( t_{1}, x\right)-\log u\left( t_{2}, y \right)\leqslant \frac{T_{x}^{2}}{4\left( t_{2}-t_{1}\right)}+\frac{1}{2}\log{\frac{t_{2}}{t_{1}}},
\]
where $T_{x}$ is defined in Definition \ref{d.6.12}.
\end{thm}

\begin{proof} The proof is standard.
Let $u\left( t, x \right):= P_{t}f \left( x \right)$ for a positive function $f \in C_{b}^{2}\left( H \right)$. Then by Theorem \ref{t.8.2}, $u$ is the solution to the heat equation
\[
\partial_{t}g=Lg, \qquad g(0,\cdot)=f.
\]

Denote $g\left( t, x \right):=\log u\left( t, x \right)$. Let $t_{2}>t_{1}\geqslant 0$, $x, y \in H$. Since $y \in V_{A}\left( x \right)$, we can find a smooth path $\gamma_{A}: [0, T_{y}] \to H^m$ such that $\gamma\left( 0 \right)=y$, $\gamma\left( T_{x} \right)=x$, and $\dot{\gamma}\left( t \right)=A\left( \gamma\left( t \right) \right)$. Define $\sigma: [0, T_{x}] \to [ t_{1}, t_{2} ] \times H^m$ by $\sigma\left( s \right):= \left(t_{2}-\frac{t_{2}-t_{1}}{T_{x}}s, \gamma\left( s \right)\right)$. Note that $\sigma\left( 0 \right)= \left( t_{2}, y \right)$ and $\sigma\left( T_{x} \right)= \left( t_{1}, x \right)$. Then
\begin{align*}
 g\left( t_{1}, x\right)&-g\left( t_{2}, y \right)=g\left(\sigma\left( 0 \right)\right)-g\left( \sigma\left( T_{x} \right)\right)
\\
&= \int_{0}^{T_{x}} \frac{d}{ds}g\left( \sigma\left( s \right)\right)\,ds
\\
&=\int_{0}^{T_{x}} \left( \langle \nabla g, \dot{\gamma_{A}} \rangle- \left( \frac{t_{2}-t_{1}}{T_{x}}\right)\partial_{t}g\left( \sigma\left( s \right)\right)\right) \,ds
\\
&\leqslant \int_{0}^{T_{x}}\nabla_{A} f|_{\gamma_{A}\left( s \right)} \,ds -\int_{0}^{T_{x}}\frac{t_{2}-t_{1}}{T_{x}}\Gamma\left( g \right)\\
&\qquad +\frac{1}{2}\int_{0}^{T_{x}} \frac{\left( t_{2}-t_{1}\right)}{T_{x}t_{2}-\left( t_{2}-t_{1}\right)s}\,ds
\end{align*}
by Corollary \ref{cor.8.8}. Note that $\Gamma\left( g \right)=\vert \nabla_{A} g\vert^{2}$, so
\begin{align*}
& \nabla_{A} f-\frac{t_{2}-t_{1}}{T_{x}}\Gamma\left( g \right)\leqslant \frac{T_{x}}{4\left( t_{2}-t_{1}\right)},
\end{align*}
 where we used the elementary estimate
 $ax-bx^{2}\leqslant {a^{2}}/{4b}$ for  $b>0$
with $x=\nabla_{A} g$. Finally, we have
\[
g\left( t_{1}, x\right)-g\left( t_{2}, y \right)\leqslant \frac{T_{x}^{2}}{4\left( t_{2}-t_{1}\right)}+\frac{1}{2}\log{\frac{t_{2}}{t_{1}}}.
\]

\end{proof}

\bibliographystyle{amsplain}
\providecommand{\bysame}{\leavevmode\hbox to3em{\hrulefill}\thinspace}
\providecommand{\MR}{\relax\ifhmode\unskip\space\fi MR }
\providecommand{\MRhref}[2]{%
  \href{http://www.ams.org/mathscinet-getitem?mr=#1}{#2}
}
\providecommand{\href}[2]{#2}

\end{document}